\documentclass[a4paper,11pt,twoside]{amsart}

\usepackage{amsmath}
\usepackage{amsthm}
\usepackage{amssymb}
\usepackage[all]{xy}
\usepackage{hyperref}
\usepackage{rotating}

\newtheorem{thm}{Theorem}[section]
\newtheorem{prop}[thm]{Proposition}
\newtheorem{lem}[thm]{Lemma}
\newtheorem{cor}[thm]{Corollary}
\newtheorem{conjecture}[thm]{Conjecture}

\numberwithin{equation}{section}

\theoremstyle{definition}
\newtheorem{definition}[thm]{Definition}
\newtheorem{remark}[thm]{Remark}
\newtheorem{ex}[thm]{Example}

\newcommand{\im}{\operatorname{im}}

\newcommand{\Db}{{\rm D}^{\rm b}}

\newcommand{\Br}{{\rm Br}}
\newcommand{\NS}{{\rm NS}}
\newcommand{\Pic}{{\rm Pic}}
\newcommand{\ch}{{\rm ch}}

\newcommand{\rk}{{\rm rk}}
\newcommand{\coh}{{\cat{Coh}}}

\newcommand{\Hom}{{\rm Hom}}

\newcommand{\Stab}{{\rm Stab}}
\renewcommand{\dim}{{\rm dim}\,}

\renewcommand{\Re}{{\rm Re}}
\renewcommand{\Im}{{\rm Im}}

\newcommand{\coker}{{\rm coker}}

\newcommand{\id}{{\rm id}}

\newcommand{\dual}{^{\vee}}
\newcommand{\mono}{\hookrightarrow}
\newcommand{\epi}{\twoheadrightarrow}
\newcommand{\mor}[1][]{\xrightarrow{#1}}
\newcommand{\isomor}{\mor[\sim]}

\newcommand{\cat}[1]{\begin{bf}#1\end{bf}}
\newcommand{\Ext}{{\rm Ext}}

\newcommand{\HH}{\mathbb{H}}

\newcommand{\fu}{\Xi}
\newcommand{\cd}[1]{{#1}^\fu}

\newcommand{\cal}{\mathcal}

\newcommand{\kb}{{\cal B}}

\newcommand{\ke}{{\cal E}}
\newcommand{\kf}{{\cal F}}
\newcommand{\kg}{{\cal G}}
\newcommand{\kh}{{\cal H}}

\newcommand{\ki}{{\cal I}}

\newcommand{\kn}{{\cal N}}

\newcommand{\ko}{{\cal O}}
\newcommand{\kp}{{\cal P}}
\newcommand{\kq}{{\cal Q}}

\newcommand{\kt}{{\cal T}}

\newcommand{\NN}{\mathbb{N}}
\newcommand{\ZZ}{\mathbb{Z}}
\newcommand{\QQ}{\mathbb{Q}}
\newcommand{\RR}{\mathbb{R}}
\newcommand{\CC}{\mathbb{C}}

\newcommand{\PP}{\mathbb{P}}

\newcommand{\mc}[1]{{\cal #1}}

\begin{document}

\title{A categorical invariant for cubic threefolds}

\author[M.\ Bernardara, E.\ Macr\`\i, S.\ Mehrotra, and P.\ Stellari]{Marcello Bernardara, Emanuele Macr\`\i, Sukhendu Mehrotra, and Paolo Stellari}

\address{M.B.: Institut de Math\'{e}matiques de Toulouse (IMT), 118 route de Narbonne, F-31062 Toulouse Cedex 9, France}
\email{marcello.bernardara@uni-due.de}

\address{E.M.: Department of Mathematics, University of Utah, 155 South 1400 East, Salt Lake City, UT 84112, USA \& Mathematical Institute, University of Bonn, Endenicher Allee 60, D-53115 Bonn, Germany}
\curraddr{Department of Mathematics, The Ohio State University, 231 W 18th Avenue, Columbus, OH 43210, USA}
\email{macri.6@math.osu.edu}

\address{S.M.: Department of Mathematics, University of Wisconsin--Madison, 480 Lincoln Drive, Madison, WI 53706-1388, USA}\email{mehrotra@math.wisc.edu }

\address{P.S.: Dipartimento di Matematica ``F. Enriques'',
Universit{\`a} degli Studi di Milano, Via Cesare Saldini 50, 20133
Milano, Italy} \email{paolo.stellari@unimi.it}



\begin{abstract}
We prove a categorical version of the Torelli theorem for cubic threefolds. More precisely, we show that the non-trivial part of a semi-orthogonal decomposition of the derived category of a cubic threefold characterizes its isomorphism class.
\end{abstract}

\maketitle

\section{Introduction}\label{sec:intro}

One of the main ideas in derived category theory, which goes back to Bondal and Orlov, is that the bounded derived category $\Db(X)$ of coherent sheaves on a smooth projective variety $X$ should contain interesting information about the geometry of the variety itself, for example, about its birational properties.

The main problem is that such information is encoded in a rather inexplicit way. A general belief is that part of this information can be obtained by looking at semi-orthogonal decompositions (see Definition \ref{def:semiorth})
\[
\Db(X)=\langle\cat{T}_1,\ldots,\cat{T}_n\rangle,
\]
where the $\cat{T}_i$'s are full triangulated subcategories of $\Db(X)$ satisfying some orthogonality conditions.
In many interesting geometric situations, all the categories $\cat{T}_i$ but one are equivalent to the derived category of a point. In the easiest case of projective spaces, one can obtain decompositions in which all the subcategories $\cat{T}_i$ are of this form.
In general, a non-trivial subcategory is present and carries useful information about $X$. This happens, for example, when $X$ is the intersection of two quadrics of even dimension \cite{BO2}. In that case, the non-trivial subcategory is equivalent to $\Db(C)$, where $C$ is the curve which is the fine moduli space of spinor bundles on the pencil generated by the two quadrics.

The same strategy has been pursued by Kuznetsov in a series of papers to study the derived categories of Fano threefolds.
In particular, an interesting example is when $X$ is a $V_{14}$ Fano threefold, i.e., a smooth complete intersection of $\PP^9$ and the Grassmannian $\mathrm{Gr}(2,6)$ in $\PP^{14}$. These are the Fano threefolds with $\Pic(X)=\ZZ$, index $1$, and genus $8$.
A classical construction shows that there is a correspondence between birational classes of $V_{14}$ Fano threefolds and isomorphism classes of cubic threefolds (i.e., smooth hypersurfaces of degree $3$ in $\PP^4$). Let $X$ be a Fano threefold as above and let $Y$ be the cubic threefold related to $X$ by the previous correspondence.
By \cite{Kuz1}, we have a semi-orthogonal decomposition of $\Db(Y)$ as
\[
\Db(Y)=\langle \cat{T}_Y,\ko_Y,\ko_Y(1)\rangle,
\]
where $\ko_Y(1):=\ko_{\PP^4}(1)|_{Y}$ and $\cat{T}_Y:=\langle\ko_Y,\ko_Y(1)\rangle^\perp$. Moreover, $\cat{T}_Y$ is an example of a Calabi--Yau category of dimension $5/3$: the composition of three copies of the Serre functor $S_{\cat{T}_Y}$ is isomorphic to the shift by $5$.
Furthermore, again by \cite{Kuz1}, the category $\Db(X)$ has a semi-orthogonal decomposition also, whose unique non-trivial part $\cat{T}_X$ is equivalent to $\cat{T}_Y$.
As a consequence, Kuznetsov deduces that $\cat{T}_X$ is a birational invariant for $V_{14}$ Fano threefolds.
A natural question is now whether $\cat{T}_X$ is a ``good'' invariant which characterizes the birational class of $X$.

By the correspondence mentioned above, one can then forget about $X$ and just study the cubic hypersurface $Y$. In particular, the problem of $\cat{T}_X$ being a good invariant translates into the problem of  $\cat{T}_Y$ being a good isomorphism invariant of $Y$.
The main result of this paper (conjectured by Kuznetsov in \cite{Kuz1}) gives a complete answer to this question and can be regarded as the categorical version of the classical Clemens-Griffiths-Tyurin Torelli Theorem for cubic threefolds:

\begin{thm}\label{conj:Kuzne3folds}
Two cubic threefolds $Y_1$ and $Y_2$ are isomorphic if and only if $\cat{T}_{Y_1}$ and $\cat{T}_{Y_2}$ are equivalent triangulated categories.
\end{thm}

Since the Picard group of a cubic threefold is free of rank one, an isomorphism $Y_1\cong Y_2$ clearly implies that $\cat{T}_{Y_1}\cong\cat{T}_{Y_2}$.
To prove the non-trivial implication of Theorem \ref{conj:Kuzne3folds}, we follow an idea of Kuznetsov \cite{Kuz1}, which can be roughly summarized as follows.
The ideal sheaves $I_l$ of lines $l$ in $Y$ are all inside the category $\cat{T}_Y$.
We define a stability condition on $\cat{T}_Y$ in such a way that any stable object in $\cat{T}_Y$ numerically equivalent to $I_l$ is actually isomorphic to some $I_{l'}$.
The Fano surface of lines $F(Y)$ of $Y$ is thus realized as a moduli space of stable objects in $\cat{T}_Y$ and it is possible to reconstruct the intermediate Jacobian $J(Y)$ of $Y$ from $\cat{T}_Y$ (being $J(Y)$ the Albanese variety of $F(Y)$).
Theorem \ref{conj:Kuzne3folds} follows from this as a consequence of the Torelli Theorem for cubic threefolds \cite{CG}.

\medskip

Theorem \ref{conj:Kuzne3folds} can be also interpreted as supporting evidence for a conjecture of Kuznetsov about the rationality of cubic fourfolds (i.e., smooth hypersurfaces of degree $3$ in $\PP^5$). If $Y$ is such a fourfold, the results of \cite{Kuz1} yield another semi-orthogonal decomposition
\[
\Db(Y)=\langle \cat{T}_Y,\ko_Y,\ko_Y(1),\ko_Y(2)\rangle,
\]
where $\cat{T}_Y$ is a Calabi--Yau category of dimension $2$. Kuznetsov conjectures in \cite{Kuz3} that this category encodes a fundamental geometric property of the fourfold:

\begin{conjecture}\label{conj:Kuzne4folds} {\bf (Kuznetsov)}
A cubic fourfold $Y$ is rational if and only if $\cat{T}_Y$ is equivalent to the derived category of a K3 surface.
\end{conjecture}

This conjecture can be regarded as a categorical counterpart in higher dimension of the construction of the intermediate Jacobian for a threefold.
Indeed, according to \cite{CG}, the intermediate Jacobian becomes a birational invariant once one forgets all its irreducible factors coming from curves (and points).
The category $\cat{T}_Y$ should have precisely the same flavor: after discarding the part coming from points, curves and surfaces, what remains is expected to be a birational invariant \cite{Kuz3}. Indeed, in the rational case, the category $\cat{T}_Y$ is not a birational invariant (see, for example, Proposition \ref{lem:rat}), while the expected birational invariant is trivial.

In the threefold case, Theorem \ref{conj:Kuzne3folds} fits this picture: it can be interpreted by saying that the category $\cat{T}_Y$ carries precisely the same information as the intermediate Jacobian of $Y$.

The main evidence for Conjecture \ref{conj:Kuzne4folds} is that, in \cite{Kuz3}, it has been verified for all the basic examples of rational cubic fourfolds: the Pfaffian cubics, the cubic fourfolds containing a plane described in \cite{H1} and the singular cubics. If the previous conjecture is true, then it would follow from the calculations in the appendix of \cite{Kuz3} that, generically, cubic fourfolds are not rational. In Section \ref{sec:example4folds} we will clarify how an analogue of Theorem \ref{conj:Kuzne3folds} can be stated for cubic fourfolds containing a plane.

\medskip

The plan of the paper is as follows. Sections \ref{sec:prelim} and \ref{sec:stab-cond} are mainly devoted to the construction of a bounded $t$-structure and a stability condition on $\cat{T}_Y$. This is the first step toward the construction of a suitable stability condition on this category. Many ingredients will play a role in these two sections, among which are Kuznetsov's results about quadric fibrations, the description of the numerical Grothendieck group of $\cat{T}_Y$, and the definition of a slope-stability for sheaves on the projective plane which are also modules over a certain algebra.

The definition is quite natural but what requires much more work is to show that the ideal sheaves of lines in $Y$ (which are all objects of $\cat{T}_Y$) are all stable in this stability condition and that they are (up to even shifts) the only stable objects in their numerical class. This is achieved in Section \ref{sect:stabid}. In Section \ref{sect:finalproof} we then finish the proof of Theorem \ref{conj:Kuzne3folds}. To this end, in Section \ref{subsect:Fanoprelim}, we use the stability results to first show that any equivalence $\cat{T}_Y\cong\cat{T}_{Y'}$ yields a bijection between the sets of ideal sheaves of lines (and so between the Fano surfaces of $Y$ and $Y'$). A rather technical argument involving the homological nature of the equivalence allows us to observe that this bijection is a morphism (Section \ref{subsect:convuniv}).
The result then follows from the Torelli theorem for cubic threefolds (Section \ref{subsect:isomor}).

We conclude the paper in Section \ref{sec:example4folds} by studying the case of cubic fourfolds with a plane. The analogies with Theorem \ref{conj:Kuzne3folds} are examined.

\bigskip

\noindent{\bf Notation.} In this paper all varieties are defined over the complex numbers $\CC$ and all triangulated categories are assumed to be essentially small (i.e., equivalent to a small category) and linear over $\CC$ (i.e., all Hom-spaces are $\CC$-vector spaces).
For a variety $X$, $\Db(X):=\Db(\coh(X))$ is the bounded derived category of coherent sheaves on $X$.
All derived functors will be denoted as if they were underived, e.g.\ for a morphism of varieties $f:X\to Y$, we will denote $f^*$ for the derived pull-back, $f_*$ for the derived push-forward, and so on.
For a complex number $z\in\CC$, we will write $\Re(z)$, resp.\ $\Im(z)$, for the real, resp.\ imaginary, part of $z$.

\section{Preliminaries}\label{sec:prelim}

In this section we realize the category $\cat{T}_Y$, for $Y$ a cubic threefold, as a full subcategory of the derived category $\Db(\PP^2,\kb_0)$ of sheaves on $\PP^2$ with an action of an algebra $\kb_0$. We describe the numerical Grothendieck groups of $\cat{T}_Y$ and of $\Db(\PP^2,\kb_0)$, and a notion of $\mu$-stability for $\kb_0$-sheaves on $\PP^2$. The construction of a stability condition on $\cat{T}_Y$ in Section \ref{sec:stab-cond} will rely on this realization.

\subsection{Kuznetsov's theorem on quadric fibrations}\label{subsec:KuzneQuadrics}

We start by recalling the notion of semi-orthogonal decomposition (see \cite{Bondal}).
Let $\cat{D}$ be a triangulated category.

\begin{definition}\label{def:semiorth}
A \emph{semi-orthogonal} decomposition of $\cat{D}$ is a sequence of full triangulated subcategories $\cat{T}_1,\ldots,\cat{T}_n\subseteq\cat{D}$ such that $\Hom_{\cat{D}}(\cat{T}_i,\cat{T}_j)=0$, for $i>j$ and, for all $K\in\cat{D}$, there exists a chain of morphisms in $\cat{D}$
\[
0=K_n\to K_{n-1}\to\ldots\to K_1\to K_0=K
\]
with $\mathrm{cone}(K_i\to K_{i-1})\in\cat{T}_i$, for all $i=1,\ldots,n$.
We will denote such a decomposition by $\cat{D}=\langle\cat{T}_1,\ldots,\cat{T}_n\rangle$.
\end{definition}

The easiest examples of semi-orthogonal decompositions are constructed via exceptional objects.

\begin{definition}\label{def:excobjgeneral}
An object $E\in\cat{D}$ is called \emph{exceptional} if $\Hom_{\cat{D}}(E,E)\cong\CC$ and $\Hom_{\cat{D}}^p(E,E)=0$, for all $p\neq0$.
A collection $\{E_1,\ldots,E_m\}$ of objects in $\cat{D}$ is called an \emph{exceptional collection} if $E_i$ is an exceptional object, for all $i$, and $\Hom_{\cat{D}}^p(E_i,E_j)=0$, for all $p$ and all $i>j$.
\end{definition}

If the category $\cat{D}$ has good properties (for example if it is equivalent to the bounded derived category of coherent sheaves on a smooth projective variety) and $\{E_1,\ldots,E_m\}$ is an exceptional collection in $\cat{D}$, then we have a semi-orthogonal decomposition
\[
\cat{D}=\langle\cat{T},E_1,\ldots,E_m\rangle,
\]
where, by abuse of notation, we denoted by $E_i$ the triangulated subcategory generated by $E_i$ (equivalent to the bounded derived category of finite dimensional vector spaces) and
\[
\cat{T}:=\langle E_1,\ldots,E_m\rangle^\perp=\left\{K\in\cat{T}\,:\,\Hom^p(E_i,K)=0,\text{ for all }p\text{ and }i\right\}.
\]

\bigskip

Let $Y$ be a smooth hypersurface in $\PP^4$ defined by a polynomial of degree 3 and let $\Db(Y)$ be its bounded derived category of coherent sheaves.
Define $\ko_Y(1):=\ko_{\PP^4}(1)|_{Y}$ and $\cat{T}_Y:=\langle\ko_Y,\ko_Y(1)\rangle^\perp$.
Since the collection $\{\ko_Y,\ko_Y(1)\}$ is exceptional in $\Db(Y)$, we have a semi-orthogonal decomposition of $\Db(Y)$ as
\[
\Db(Y)=\langle\cat{T}_Y,\ko_Y,\ko_Y(1)\rangle.
\]

\begin{remark}
As Kuznetsov observed in \cite{Kuz1}, the sheaves of ideals $I_l$ of lines $l$ in $Y$ are all inside the category $\cat{T}_Y$.
Moreover, although this will not be used in this paper, all instanton bundles (and their twist by $\ko_Y(1)$) are in $\cat{T}_Y$ as well.
\end{remark}

Fix a line $l_0$ inside $Y$. Consider the blow-up $\widetilde{\PP^4}$ of $\PP^4$ along $l_0$ and $q: \widetilde{\PP^4} \to \PP^2$ the $\PP^2$-bundle induced by the projection from $l_0$ onto a plane. Let $\widetilde{Y}$ denote the strict transform of $Y$ via this blow-up. The restriction of $q$ to $\widetilde{Y}$ induces a conic fibration $\pi: \widetilde{Y} \to \PP^2$, so that we have the following diagram:
\[
\xymatrix{\PP^4 && \widetilde{\PP^4}={\mathrm{Bl}}_{l_0} \PP^4 \ar[ll] \ar[rrd]^{q} &&\\
Y \ar@{^{(}->}[u] && \widetilde{Y}={\mathrm{Bl}}_{l_0} Y \ar[ll]_{\sigma} \ar[rr]^{\pi} \ar@{^{(}->}[u]^{\alpha} && \PP^2.
}
\]
We let $D \subset \widetilde{Y}$ be the exceptional divisor of the blow-up $\sigma: \widetilde{Y} \to Y$.
By abuse of notation, we denote by $h$ both the class of a line in $\PP^2$ and its pull-backs to $\widetilde{\PP^4}$ and $\widetilde{Y}$. In the same way, we denote by $H$ both the class of a hyperplane in $\PP^4$ and its pull-backs to $\widetilde{\PP}^4$ and $\widetilde{Y}$.
For later use, we notice that the calculation in \cite[Lemma 4.1]{Kuz3}, adapted to the cubic threefolds case, yields $\ko_{\widetilde{Y}}(D)\cong\ko_{\widetilde{Y}}(H-h)$, $\widetilde{\PP}^4 \cong \PP_{\PP^2} (\ko_{\PP^2}^{\oplus 2} \oplus \ko_{\PP^2}(-h))$, the relative ample line bundle is $\ko_{\widetilde{\PP}^4}(H)$, and the relative canonical bundle is $\ko_{\widetilde{\PP}^4}(h-3H)$.

Let $\kb_0$ (resp.\ $\kb_1$) be the sheaf of even (resp.\ odd) parts of the Clifford algebra corresponding to $\pi$, as in \cite[Sect.\ 3]{Kuz2}.
Explicitly, in our case,
\begin{align}\label{bcoh}
\kb_0&\cong\ko_{\PP^2}\oplus\ko_{\PP^2}(-h)\oplus\ko_{\PP^2}(-2h)^{\oplus 2}\\
\kb_1&\cong\ko_{\PP^2}^{\oplus 2}\oplus\ko_{\PP^2}(-h)\oplus\ko_{\PP^2}(-2h),\nonumber
\end{align}
as sheaves of $\ko_{\PP^2}$-modules. Denote by $\coh(\PP^2,\kb_0)$ the abelian category of right coherent $\kb_0$-modules and by $\Db(\PP^2,\kb_0)$ its bounded derived category.

By \cite[Sect.\ 4]{Kuz2}, we can define a fully faithful functor $\Phi:=\Phi_{\ke'}:\Db(\PP^2,\kb_0)\to\Db(\widetilde{Y})$, $\Phi_{\ke'}(M):=\pi^*M\otimes_{\pi^*\kb_0}\ke'$, for all $M\in\Db(\PP^2,\kb_0)$, where $\ke'\in\coh(\widetilde{Y})$ is a rank $2$ vector bundle on $\widetilde{Y}$ with a natural structure of flat left $\pi^*\kb_0$-module.
We will not need the actual definition of $\ke'$ (for which the reader is referred to \cite[Sect.\ 4]{Kuz2}) but only the presentation
\begin{equation}\label{eq:defofkeprime}
0\to q^*\kb_0(-2H)\to q^*\kb_1(-H)\to\alpha_*\ke'\to0.
\end{equation}

The left and right adjoint functors of $\Phi$ are respectively
\begin{align}\label{bascoh}
\Psi(-)&:=\pi_*((-)\otimes\ko_{\widetilde{Y}}(h)\otimes\ke[1]),\\
\Pi(-)&:=\pi_*(\kh om_{\widetilde{Y}}(\ke',-)),\nonumber
\end{align}
where $\ke\in\coh(\widetilde{Y})$ is another rank $2$ vector bundle on $\widetilde{Y}$ with a natural structure of right $\pi^*\kb_0$-module (see \cite[Sect.\ 4]{Kuz2}). The main property we will need is the presentation
\begin{equation}\label{eq:defofke}
0\to q^*\kb_1(-h-2H)\to q^*\kb_0(-H)\to\alpha_*\ke\to0.
\end{equation}

\medskip

The embedding $\Phi$ has the remarkable property that, by \cite[Thm.\ 4.2]{Kuz2}, it gives a semi-orthogonal decomposition of $\Db(\widetilde{Y})$ as
\begin{equation}\label{eq:1}
\Db(\widetilde{Y})=\langle\Phi(\Db(\PP^2,\kb_0)),\ko_{\widetilde{Y}}(-h),\ko_{\widetilde{Y}},\ko_{\widetilde{Y}}(h)\rangle.
\end{equation}

On the other hand, a well-known result of Orlov \cite{Orlov} tells us that
\begin{equation}\label{eq:2}
\Db(\widetilde{Y})=\langle\sigma^*\cat{T}_Y,\ko_{\widetilde{Y}},\ko_{\widetilde{Y}}(H),\ko_D,\ko_D(H)\rangle,
\end{equation}
where $\sigma^*:\Db(Y)\to\Db(\widetilde{Y})$ is fully faithful.

We now perform some mutations to compare $\cat{T}_Y$ and $\Db(\PP^2,\kb_0)$, by mimicking \cite[Sect.\ 4]{Kuz3}.
The aim of that is to identify $\cat{T}_Y$ with an admissible subcategory of $\Db(\PP^2,\kb_0)$. This will allow
us to use the properties of sheaves of $\kb_0$-algebras on $\PP^2$ to describe a stability condition on $\cat{T}_Y$.
We first recall some basics about mutations (see \cite{Bondal}).

Let $E$ be an exceptional object in a triangulated category $\cat{D}$ with good properties as before (for us, $\Db(\widetilde{Y})$).
Consider the two functors, respectively \emph{left and right mutation}, $L_E,R_E:\cat{D}\to\cat{D}$ defined by
\begin{align*}
L_E(M)&:=\mathrm{cone}\left(ev:\mathrm{RHom}(E,M)\otimes E\to M\right)\\
R_E(M)&:=\mathrm{cone}\left(ev^\vee:M\to\mathrm{RHom}(M,E)^\vee\otimes E\right)[-1].
\end{align*}
The main property of mutations is that, given a semi-orthogonal decomposition of $\cat{D}$
\[
\langle\cat{T}_1,\ldots,\cat{T}_k,E,\cat{T}_{k+1},\ldots,\cat{T}_n\rangle,
\]
we can produce two new semi-orthogonal decompositions
\[
\langle\cat{T}_1,\ldots,\cat{T}_k,L_E(\cat{T}_{k+1}),E,\cat{T}_{k+2},\ldots,\cat{T}_n\rangle
\]
and
\[
\langle\cat{T}_1,\ldots,\cat{T}_{k-1},E,R_E(\cat{T}_k),\cat{T}_{k+1},\ldots,\cat{T}_n\rangle.
\]

Coming back to $\Db(\widetilde{Y})$, we first mutate the pair $\left(\Phi(\Db(\PP^2,\kb_0)),\ko_{\widetilde{Y}}(-h)\right)$ in \eqref{eq:1}, to obtain a new pair $\left(\ko_{\widetilde{Y}}(-h),\Phi'(\Db(\PP^2,\kb_0))\right)$, with $\Phi':=R_{\ko_{\widetilde{Y}}(-h)}\circ\Phi$.
Then, by \cite[Lem.\ 2.11]{Kuz3}, using the fact that the canonical bundle of $\widetilde{Y}$ is $\ko_{\widetilde{Y}}(-H-h)$, we obtain a new semi-orthogonal decomposition
\begin{equation}\label{eq:3}
\Db(\widetilde{Y})=\langle\Phi'(\Db(\PP^2,\kb_0)),\ko_{\widetilde{Y}},\ko_{\widetilde{Y}}(h),\ko_{\widetilde{Y}}(H)\rangle.
\end{equation}
On the other side, it is easy to see that the pair $\left(\ko_{\widetilde{Y}}(H),\ko_D\right)$ is completely orthogonal and that the left mutation of the pair $\left(\ko_{\widetilde{Y}}(mH),\ko_D(mH)\right)$ is $\left(\ko_{\widetilde{Y}}(h+(m-1)H),\ko_{\widetilde{Y}}(mH)\right)$ for all integers $m$.
Hence, from \eqref{eq:2}, we obtain a new semi-orthogonal decomposition
\begin{equation}\label{eq:4}
\Db(\widetilde{Y})=\langle\sigma^*\cat{T}_Y,\ko_{\widetilde{Y}}(h-H),\ko_{\widetilde{Y}},\ko_{\widetilde{Y}}(h),\ko_{\widetilde{Y}}(H)\rangle.
\end{equation}
By comparing the two semi-orthogonal decompositions \eqref{eq:3} and \eqref{eq:4}, we obtain a semi-orthogonal decomposition
\begin{equation}\label{eqn:last}
\Phi'(\Db(\PP^2,\kb_0))=\langle\sigma^*\cat{T}_Y,\ko_{\widetilde{Y}}(h-H)\rangle.
\end{equation}

We finally get an equivalence between $\cat{T}_Y$ and an admissible subcategory of $\Db(\PP^2,\kb_0)$. We can now
consider objects of $\cat{T}_Y$ as complexes of $\kb_0$-algebras on $\PP^2$ through the fully faithful functor
\[
\fu:=(\sigma_* \circ \Phi')^{-1}:\cat{T}_Y\longrightarrow\Db(\PP^2,\kb_0).
\]
Nevertheless, it will be convenient to specify the ambient category. Thus,
if $A$ is an object of $\cat{T}_Y$, we set
\[
\cd{A}:=\fu(A)\in\Db(\PP^2,\kb_0).
\]

\begin{ex}\label{ex:idealsheaves}
As an illustration of the previous procedure, we compute the image $\cd{I_l}$ in $\Db(\PP^2,\kb_0)$ of an ideal sheaf
$I_l$ of a line $l$ in $Y$ which does not intersect $l_0$.
The preliminary step is the following easy computation:
\begin{equation}\label{eq:interestingformula}
\Psi(\ko_{\widetilde{Y}}(mh))=0,
\end{equation}
for all integers $m$.
Indeed, by the projection formula, \eqref{eq:defofke}, and the fact that $\pi = q \circ \alpha$,
\begin{equation*}
\begin{split}
\Psi(\ko_{\widetilde{Y}}(mh))&\cong\Psi(\alpha^*(\ko_{\widetilde{\PP^4}}(mh)))\\
               &\cong q_*(\ko_{\widetilde{\PP^4}}((m+1)h)\otimes\alpha_*\ke[1])=0,
\end{split}
\end{equation*}
for all $m$. Here we are using the fact that $\widetilde{\PP^4}\to\PP^2$ is a projective bundle and $\ko_{\widetilde{\PP^4}/\PP^2}(1)\cong\ko_{\widetilde{\PP^4}}(H)$.

Now, to compute $\cd{I_l}=(\sigma_*\circ\Phi')^{-1}(I_l)$, we first notice that, by \eqref{eq:interestingformula}, the mutation by $\ko_{\widetilde{Y}}(-h)$ has no effect.
More precisely, $\cd{I_l}\cong\Psi(\sigma^*I_l)$.
Now the rational map $Y\dashrightarrow\PP^2$ is well-defined on $l$ and maps it to another line; denote by $j$ the embedding $l\mono\widetilde{Y}\mor[\pi]\PP^2$.
Pulling back the exact sequence
\[
0\to I_l\to\ko_Y\to\ko_l\to0,
\]
we have another exact sequence
\[
0\to\sigma^*I_l\to\ko_{\widetilde{Y}}\to\sigma^*\ko_l=\ko_l\to0.
\]
Again, by \eqref{eq:interestingformula}, we have $\Psi(\ko_{\widetilde{Y}})=0$ and so
\begin{equation*}
\begin{split}
\cd{I_l}\cong\Psi(\sigma^*I_l)&\cong\Psi(\ko_l)[-1]\\
    &=\pi_*(\ko_l[-1]\otimes\ko_{\widetilde{Y}}(h)\otimes\ke[1]))\\
    &\cong j_*(j^*\ko_{\widetilde{Y}}(h)\otimes\ke |_l)\\
    &\cong j_*(\ke |_l)\otimes\ko_{\PP^2}(h).
\end{split}
\end{equation*}
\end{ex}

\subsection{Basic properties}\label{subsec:basic}

In order to define a stability condition on $\cat{T}_Y$, we need the description of its numerical Grothendieck group,
and a notion of $\mu$-stability on it.
In this section we collect some basic results on
$\cat{T}_Y$ and $\Db(\PP^2,\kb_0)$ and describe their numerical Grothendieck groups. Having described $\cat{T}_Y$ as
an admissible sucategory of $\Db(\PP^2,\kb_0)$ turns out to be a fundamental step.

Let $\cat{D}$ be a triangulated category which arises as a subcategory in a semi-orthogonal decomposition of the bounded derived category of coherent sheaves on a smooth projective variety.
In particular, $\cat{D}$ is $\Ext$-finite (i.e., its $\Hom$-spaces are finite dimensional over $\CC$), the Euler characteristic
\[
\chi(-,-):=\sum_i(-1)^i\hom^i(-,-)
\]
is well-defined (where $\hom^i(-,-):=\dim_{\CC}\Hom(-,-[i])$), and it has a Serre functor (i.e., an autoequivalence $S:\cat{D}\isomor\cat{D}$ with functorial isomorphisms $\Hom(A,B)\cong\Hom(B,S(A))^\vee$, for all $A,B\in\cat{D}$).

Denote by $K(\cat{D})$ the Grothendieck group of $\cat{D}$.

\begin{definition}
A class $[A]\in K(\cat{D})$ is \emph{numerically trivial} if $\chi([M],[A])=0$, for all $[M]\in K(\cat{D})$.
Define the numerical Grothendieck group $\kn(\cat{D})$ as the quotient of $K(\cat{D})$ by numerically trivial classes.
\end{definition}

When $\cat{D}=\Db(X)$, for $X$ a smooth projective variety, we will denote $\kn(\cat{D})$ by $\kn(X)$.

\begin{lem}\label{lem:injectivityN}
Assume we have a semi-orthogonal decomposition $\cat{D}=\langle\cat{T},E\rangle$, with $E$ an exceptional object in $\cat{D}$.
Then
\[
\kn(\cat{T})\cong\left\{[M]\in\kn(\cat{D})\,:\,\chi([E],[M])=0\right\}
\]
and $\kn(\cat{D})\cong\kn(\cat{T})\oplus\ZZ[E]$.
\end{lem}

\begin{proof}
The inclusion functor $\cat{T}\mono\cat{D}$ induces a morphism $K(\cat{T})\to K(\cat{D})\to\kn(\cat{D})$.
If $[A]\in K(\cat{T})$ is such that $\chi(K(\cat{T}),[A])=0$, we have, for any $[K] \in K(\cat{D})$,
\[
\chi([K],[A])=\chi([K_T],[A])+n\chi([E],[A])=0,
\]
since every element $[K]\in K(\cat{D})$ can be written in $K(\cat{D})$ as $[K_T]+n[E]$, with $n$ an integer and $[K_T] \in K(\cat{T})$.
As a consequence we have an injective induced map $\kn(\cat{T})\mono\kn(\cat{D})$, whose image is contained in the set $[E]^\perp:=\left\{[M]\in\kn(\cat{D})\,:\,\chi([E],[M])=0\right\}$.
Furthermore, if $[K]\in[E]^\perp$, then from the decomposition $[K]=[K_T]+n[E]$, it follows easily that $n=0$, and so that $[K]\in\kn(\cat{T})$.
\end{proof}

\begin{prop}\label{prop:GrotGroup}
Let $Y$ be a cubic threefold. Then

{\rm (i)} $\kn(Y)\cong\ZZ^{\oplus 4}\cong\ZZ[\ko_Y]\oplus\ZZ[\ko_H]\oplus\ZZ[\ko_l]\oplus\ZZ[\ko_p]$, where $H$ is a hyperplane section, $l$ a line, and $p$ a point;

{\rm (ii)} $\kn(\cat{T}_Y)\cong\ZZ^{\oplus 2}\cong\ZZ[I_l]\oplus\ZZ([S_{\cat{T}_Y}(I_l)])$, where $S_{\cat{T}_Y}$ denotes the Serre functor of $\cat{T}_Y$. The Euler characteristic $\chi(-,-)$ on $\kn(\cat{T}_Y)$ has the form, with respect to this basis,
\[
\begin{pmatrix}
-1 & -1\\
\phantom{-}0 & -1
\end{pmatrix}.
\]
\end{prop}

\begin{proof}
For the first part, see e.g.\ \cite[Cor.\ 5.8]{Kuz5}.
Let us prove the second part.
By Lemma \ref{lem:injectivityN}, we know that $\kn(\cat{T}_Y)$ is the right orthogonal complement of the classes $[\ko_Y]$ and $[\ko_Y(1)]$. Hence $\kn(\cat{T}_Y) \cong \ZZ^{\oplus 2}$. We have $\chi(I_l,I_l) = \chi(S_{\cat{T}_Y}(I_l),S_{\cat{T}_Y}(I_l))= \chi(I_l,S_{\cat{T}_Y}(I_l))=-1$.

By \cite[Cor.\ 4.4]{Kuz1}, $S_{\cat{T}_Y}^3 \cong [5]$ and so
\[
\chi(S_{\cat{T}_Y}(I_l),I_l) = \chi(I_l,S^2_{\cat{T}_Y}(I_l)) = \chi(I_l,S^{-1}_{\cat{T}_Y}(I_l)[5]) = \chi(S_{\cat{T}_Y}(I_l),I_l[5]) = - \chi(S_{\cat{T}_Y}(I_l),I_l),
\]
which means that $\chi(S_{\cat{T}_Y}(I_l),I_l) = 0$ and $[I_l], [S_{\cat{T}_Y}(I_l)]$ form a basis for $\kn(\cat{T}_Y) \otimes \QQ$. Then any $[A] \in \kn(\cat{T}_Y)$ can be written as $a[I_l] + b[S_{\cat{T}_Y}(I_l)]$ for $a,b \in \QQ$. But $\chi([A], [I_l])=-a$ and $\chi([S_{\cat{T}_Y}(I_l)],[A])=-b$ are integers.
\end{proof}

\begin{lem}\label{prop:-1classes}
Let $[A]$ be a class in $\kn(\cat{T}_Y)$ such that $\chi([A],[A]) = -1$. Then, up to a sign, $[A]$ is
either $[I_l]$, or $[S_{\cat{T}_Y}(I_l)]$, or $[S^2_{\cat{T}_Y}(I_l)]$.
\end{lem}

\begin{proof}
By Proposition \ref{prop:GrotGroup}, a class $[A]=a[I_l] + b[S_{\cat{T}_Y}(I_l)]$ ($a,b\in\ZZ$) satisfies $\chi([A],[A])=-1$ if and only if $a^2+b^2+ab=1$.
But this is possible if and only if either $ab=-1$ (and so $(a,b)=\pm(-1,1)$, that means $[A]=\pm[S^2_{\cat{T}_Y}(I_l)]$) or $ab=0$ (and so $(a,b)=\pm(1,0),\pm(0,1)$), as wanted.
\end{proof}

\begin{prop}\label{prop:properties}

{\rm (i)} $(\Phi')^{-1}(\ko_{\widetilde{Y}}(h-H))\cong\kb_1$. Consequently,
$\fu(\cat{T}_Y)=\langle \kb_1\rangle^{\perp}$.

{\rm (ii)} Serre duality holds for $\Db(\PP^2,\kb_0)$ and the Serre functor $S_{\kb_0}$ is given by $(-)\otimes_{\kb_0}\kb_{-1}[2]$, where $\kb_{-1}:=\kb_1(-h)$.
In particular, $S_{\kb_0}(\kb_1)\cong\kb_0[2]$ and $S^2_{\kb_0}(\kb_1)\cong\kb_{-1}[4]$.
\end{prop}

\begin{proof}
(i) As in Example \ref{ex:idealsheaves}, by \eqref{eq:interestingformula} we have $(\Phi')^{-1}(\ko_{\widetilde{Y}}(h-H))\cong\Psi(\ko_{\widetilde{Y}}(h-H))$.
Then, by \eqref{eq:defofke}, we can conclude that
\begin{equation*}
\begin{split}
\Psi(\ko_{\widetilde{Y}}(h-H))&\cong\Psi(\alpha^*(\ko_{\widetilde{\PP^4}}(h-H)))\\
       &\cong q_*(\ko_{\widetilde{\PP^4}}(2h-H)\otimes\alpha_*(\ke)[1])\\
       &\cong q_*(q^*(\kb_1)(h-3H)[2])\cong\kb_1,
\end{split}
\end{equation*}
where, for the last isomorphism, we used relative Serre duality with dualizing sheaf $\ko_{\widetilde{\PP}^4}(h-3H)$. The second statement is now (\ref{eqn:last}).

(ii) The expression for the Serre functor is a standard computation using adjunction, existence of locally free resolutions, and \cite[Sect.\ 2.1]{Kuz2}, once we observe that $\kb_{-1}\cong\kb_0^\vee\otimes\omega_{\PP^2}$ (using (\ref{bascoh})), where the dual is taken with respect to the $\ko_{\PP^2}$-module structure.
For the last statement (which can also be proved by direct computation), we need to show that
$$\kb_1\otimes_{\kb_0}\kb_{-1}\cong\kb_0,$$ which is precisely \cite[Cor.\ 3.9]{Kuz2}.
\end{proof}

\begin{remark}\label{rmk:homdim}
Observe that the duality statement of the above proposition implies, in particular, that the homological dimension of $\Db(\PP^2,\kb_0)$ is $2$.
\end{remark}

\begin{ex}\label{ex:ideallineblownup}
Consider the ideal sheaf $I_{l_0}$ of the blown-up line $l_0$ in $Y$. We describe the object $\cd{I_{l_0}}$: it is
a true complex with two non-vanishing cohomologies.
We have an exact triangle
\[
\kb_0[1]\to \cd{I_{l_0}}\to\kb_1\mor[\eta]\kb_0[2]
\]
 in $\Db(\PP^2,\kb_0)$, where $\eta$ is the map corresponding to the identity of $\kb_1$ via Serre duality (recall that $\Hom(\kb_1,\kb_1)\cong\CC$).

Indeed, first of all, by \cite[Prop.\ 11.12]{Huy}, we have an exact triangle
\[
\ko_D(D)[1]\to\sigma^*(\ko_{l_0})\to\ko_D
\]
in $\Db(\widetilde{Y})$, where as before $D$ denotes the exceptional divisor of $\widetilde{Y}$.
By \eqref{eq:interestingformula}, to compute $\cd{I_{l_0}}$ it is sufficient to compute $\Psi(\ko_D)$ and $\Psi(\ko_D(D))$.
To this end, we simply use \eqref{eq:defofke} and the exact sequence
\[
0\to\ko_{\widetilde{Y}}(h-H)\to\ko_{\widetilde{Y}}\to\ko_D\to 0.
\]
We have
\begin{align*}
\Psi(\ko_D)&\cong\Psi(\ko_{\widetilde{Y}}(h-H))[1]\cong\kb_1[1]\\
\Psi(\ko_D(D))&\cong\Psi(\ko_{\widetilde{Y}}(H-h))\cong\kb_0[1].
\end{align*}
Hence, we have an exact triangle
\[
\kb_0[2]\to\Psi(\sigma^*\ko_{l_0})\to\kb_1[1].
\]
By \eqref{eq:interestingformula}, we have $\cd{I_{l_0}}\cong\Psi(\sigma^*I_{l_0})\cong\Psi(\sigma^*\ko_{l_0})[-1]$.
Hence we can write $\cd{I_{l_0}}$ as an extension of $\kb_0[1]$ by $\kb_1$.
The fact that it is the unique non-trivial extension follows from $\cd{I_{l_0}}\in\langle\kb_1\rangle^\perp$.
\end{ex}

\begin{prop}\label{prop:GrotGroupModules}
We have
\[
\kn(\PP^2,\kb_0):=\kn(\Db(\PP^2,\kb_0))\cong\ZZ^{\oplus 3}\cong\ZZ[\kb_1]\oplus\ZZ[\kb_0]\oplus\ZZ[\kb_{-1}].
\]
\end{prop}

\begin{proof}
By Lemma \ref{lem:injectivityN}, $\kn(\PP^2,\kb_0)\cong\kn(\cat{T}_Y)\oplus\ZZ[\kb_1]$.
By Example \ref{ex:ideallineblownup}, we know that $[\cd{I_l}]=[\kb_1]-[\kb_0]$.
Finally, by Proposition \ref{prop:properties} (ii), making use of the expression for
the Serre functor of an admissible category (see \cite{Bondal}),
\[
\Xi(S_{\cat{T}_Y}(-))\cong R_{\kb_0}(S_{\kb_0}(\Xi(-))).
\]
Hence
\begin{equation*}
\begin{split}
[\cd{S_{\cat{T}_Y}(I_l)}]&=[S_{\kb_0}(\cd{I_l})]-\chi(S_{\kb_0}(\cd{I_l}),\kb_0)[\kb_0]\\
     &=[\kb_0]-[\kb_{-1}]-(-1)[\kb_0]\\
     &=2[\kb_0]-[\kb_{-1}].
\end{split}
\end{equation*}
This implies, by Proposition \ref{prop:GrotGroup}, that $[\kb_1]$, $[\kb_0]$, and $[\kb_{-1}]$ form a basis of $\kn(\PP^2,\kb_0)$.
\end{proof}

\subsection{$\mu$-stability}\label{sec:mustability}

In order to define a stability condition on $\cat{T}_Y$ we need a $t$-structure and a stability function on it. To this end,
we construct a $t$-structure and a stability function on $\Db(\PP^2,\kb_0)$ and restrict them using the identification
of $\cat{T}_Y$ with an admissible subcategory of $\Db(\PP^2,\kb_0)$. Stability conditions can be defined on surfaces
by tilting the standard $t$-structure with respect to a torsion pair \cite{Br1}. In order to perform this first step on $\Db(\PP^2,\kb_0)$,
we need a notion of $\mu$-stability for sheaves of $\kb_0$-algebras.

In this section we show that there exists a notion of $\mu$-stability on $\coh(\PP^2,\kb_0)$ (or rather for objects in $\coh(\PP^2,\kb_0)$ which are torsion-free as sheaves on $\PP^2$) satisfying the following properties:

(1) Harder--Narasimhan and Jordan--H\"older filtrations in $\mu$-(semi)stable objects exist.

(2) $\Hom(K,\widetilde{K})=0$, if $K,\widetilde{K}$ are $\mu$-semistable torsion-free sheaves with $\mu(K)>\mu(\widetilde{K})$.

(3) The Serre functor preserves $\mu$-stability: If $K$ is a torsion-free $\mu$-semistable sheaf, then $K\otimes_{\kb_0}\kb_{-1}$ is $\mu$-semistable too and
\[
\mu(K\otimes_{\kb_0}\kb_{-1})<\mu(K).
\]

(4) The exceptional object $\kb_1$ is $\mu$-stable.

\medskip

We start by defining the numerical functions \emph{rank} and \emph{degree} on $\kn(\PP^2,\kb_0)$.
Consider the forgetful functor $\mathrm{Forg}:\Db(\PP^2,\kb_0)\to\Db(\PP^2)$ which forgets the structure of $\kb_0$-module.
Then $\mathrm{Forg}$ has a left adjoint $-\otimes_{\ko_{\PP^2}}\kb_0$. Hence, $\mathrm{Forg}$ induces a group homomorphism
\[
\kn(\PP^2,\kb_0)\longrightarrow\kn(\PP^2)\cong K(\PP^2)
\]
between the numerical Grothendieck groups of $\Db(\PP^2,\kb_0)$ and $\Db(\PP^2)$.
Define
\begin{align*}
\rk&:\kn(\PP^2,\kb_0)\to\ZZ,\qquad\rk(K):=\rk(\mathrm{Forg}(K))\\
\deg&:\kn(\PP^2,\kb_0)\to\ZZ,\qquad\deg(K):=\mathrm{c}_1(\mathrm{Forg}(K)).\mathrm{c}_1(\ko_{\PP^2}(h)).
\end{align*}
For $K\in\coh(\PP^2,\kb_0)$ with $\rk(K)\neq0$, define the slope $\mu(K):=\deg(K)/\rk(K)$. Moreover, when we say that $K$ is either torsion-free or torsion of dimension $d$, we always mean that $\mathrm{Forg}(K)$ has this property.


\begin{lem}\label{lem:rankpoints}
{\rm (i)} The image of the numerical function $\rk:\kn(\PP^2,\kb_0)\to\ZZ$ is $4 \ZZ$.

{\rm (ii)} Let $K \in \cat{Coh}(\PP^2, \kb_0)$. Then $\rk(K)$ is a multiple of $8$ if and only if $\deg(K)$ is even.

{\rm (iii)} Let $K \in \cat{Coh}(\PP^2, \kb_0)$ be such that $\rk(K)= \deg(K) = 0$. Then $\Hom^i(\kb_1,K)=0$, for $i\neq 0$, and $\Hom(\kb_1,K)\neq 0$.
\end{lem}

\begin{proof}
By Proposition \ref{prop:GrotGroupModules}, we have
\[
[K]=a[\kb_1]+b[\kb_0]+c[\kb_{-1}],
\]
for some integer $a,b,c$. Then $\rk(K)=4(a+b+c)$ and $\deg(K)=-3a-5b-7c$. This proves (i) and (ii).

For part (iii), applying the functor $(-)\otimes_{\kb_0}\kb_{-1}$, we have
\[
\Hom^i(\kb_1,K)\cong\Hom^i(\kb_0,K\otimes_{\kb_0}\kb_{-1})\cong\Hom^i(\ko_{\PP^2},\mathrm{Forg}(K\otimes_{\kb_0}\kb_{-1})),
\]
where the last isomorphism is again given by adjunction. Now $K':=\mathrm{Forg}(K\otimes_{\kb_0}\kb_{-1})$ is a non-trivial sheaf on $\PP^2$ with the additional conditions that $\rk(K')=\deg(K')=0$, because $\rk(K)=\deg(K)=0$.  Hence $\Hom^i(\ko_{\PP^2},K')=0$ if $i\neq 0$ while it is non-trivial if $i=0$.
\end{proof}

\begin{definition}\label{def:mustability}
An object $K\in\coh(\PP^2,\kb_0)$ such that $\mathrm{Forg}(K)$ is torsion-free is called $\mu$-(semi)stable if $\mu(L)<\mu(K)$ (resp.\ $\leq$), for all $0\neq L\mono K$ in $\coh(\PP^2,\kb_0)$ with $\rk(L)<\rk(K)$.
\end{definition}

By repeating literally the standard proofs (see, for example, \cite[Sects. 1.2, 1.3, 1.5, 1.6]{HL}), one easily shows that the $\mu$-stability we defined enjoys
properties (1) and (2). (For a more general treatment, see \cite[Sect.\ 3]{Simpson}.) Moreover, for a sheaf $K$ of $\kb_0$-modules on $\PP^2$, the decomposition in
torsion part $K_{tor}$ and torsion free part $K_{tf}$ is compatible with the $\kb_0$-module structure.

\begin{lem}\label{lem:tfreepart}
Let $K\in\coh(\PP^2,\kb_0)$. Then its torsion and torsion-free part, considered as a sheaf on $\PP^2$, have a natural structure of $\kb_0$-module such that there is an exact sequence in $\coh(\PP^2,\kb_0)$
\[
0\to K_{tor}\to K\to K_{tf}\to 0.
\]
\end{lem}

\begin{proof}
This can be easily proved by observing that the action $K\otimes_{\ko_{\PP^2}}\kb_0\to K$ maps $K_{tor}\otimes_{\ko_{\PP^2}}\kb_0$ into $K_{tor}$.
\end{proof}

Finally, the $\mu$-stability we have defined enjoys properties (3) and (4).

\begin{lem}\label{lemma:property(3)}
Let $K$ be a torsion-free, $\mu$-(semi)stable sheaf in $\coh(\PP^2,\kb_0)$. Then $K\otimes_{\kb_0}\kb_{-1}$
is $\mu$-(semi)stable and
\begin{equation}\label{eq:muSerre}
\mu(K\otimes_{\kb_0}\kb_{-1})=\mu(K)- \frac{1}{2}.
\end{equation}
\end{lem}

\begin{proof}
By Proposition \ref{prop:properties} (ii), $\widetilde{K}:=K\otimes_{\kb_0}\kb_{-1}$ is a torsion free sheaf in $\coh(\PP^2,\kb_0)$.

As in the proof of Lemma \ref{lem:rankpoints}, given $[K]=a[\kb_1]+b[\kb_0]+c[\kb_{-1}]$,
we have $\rk(K)=4(a+b+c)$, $\deg(K)=-3a-5b-7c$, $\rk(\widetilde{K})=4(a+b+c)=\rk(K)$, and $\deg(\widetilde{K})=-5a-7b-9c=\deg(K)-(1/2)\rk(K)$.
From this we immediately deduce \eqref{eq:muSerre}.
Moreover, if $A\mono\widetilde{K}$ is such that $\mu(A)>\mu(\widetilde{K})$ (resp.\ $\geq$), then $A\otimes_{\kb_0}\kb_{1}\mono K$ and $\mu(A\otimes_{\kb_0}\kb_{1})>\mu(K)$ (resp.\ $\geq$), contradicting the $\mu$-(semi)stability of $K$.
Hence $\widetilde{K}$ is $\mu$-(semi)stable, as wanted.
\end{proof}

Notice that, by Lemma \ref{lem:rankpoints} (i), every object in $\coh(\PP^2,\kb_0)$ has rank a multiple of $4$.
Since $\kb_1$ is locally-free of rank $4$ it must be $\mu$-stable, which is precisely property (4).

\section{A stability condition}\label{sec:stab-cond}

In this Section, we construct a stability condition on $\cat{T}_Y$, which we will use in Sections \ref{sect:stabid} and \ref{sect:finalproof} to
reconstruct the Fano surface of lines on $Y$ as a moduli space of objects on $\cat{T}_Y \subset \Db(Y)$. More precisely, we will prove the following.

\begin{thm}\label{thm:tstruct3}
There exists a locally finite stability condition $\sigma:=(Z,\cat{B})$ on $\cat{T}_Y$ such that, for all lines $l\subseteq Y$, the ideal sheaf $I_l$ is contained in the heart $\cat{B}$ of a bounded $t$-structure on $\cat{T}_Y$.
\end{thm}

We first recall Bridgeland's definition of a stability condition in \ref{subsec:4.1}.
Then we will use the description of the numerical
Grothendieck group and the $\mu$-stability on $\Db(\PP^2,\kb_0)$ to get first a $t$-structure and a stability
function on $\Db(\PP^2,\kb_0)$. Restricting this to $\cat{T}_Y$, seen as an admissible subcategory of $\Db(\PP^2,\kb_0)$
and using the description of its numerical Grothendieck group, we get the stability condition.

\subsection{Bridgeland's stability conditions}\label{subsec:4.1}

We recall Bridgeland's definition of the notion of stability condition on a triangulated category.
Let $\cat{T}$ be a triangulated category. A \emph{stability condition} on $\cat{T}$
is a pair $\sigma=(Z,\kp)$ where $Z:K(\cat{T})\to\CC$ is a group
homomorphism and
$\kp(\phi)\subset\cat{T}$ are full additive subcategories,
$\phi\in\RR$, satisfying the following conditions:

(a) If $0\ne C\in\kp(\phi)$, then $Z(C)=m(C)\exp(i\pi\phi)$ for
some $m(C)\in\RR_{>0}$.

(b) $\kp(\phi+1)=\kp(\phi)[1]$ for all $\phi\in\RR$.

(c) If $\phi_1>\phi_2$ and $C_i\in\kp(\phi_i)$, $i=1,2$, then
$\Hom_{\cat{T}}(C_1,C_2)=0$.

(d) Any $0\ne C\in\cat{T}$ admits a \emph{Harder--Narasimhan
filtration} (\emph{HN-filtration} for short) given by a collection
of distinguished triangles $C_{i-1}\to C_i\to A_i$ with $C_0=0$
and $C_n=C$ such that $A_i\in\kp(\phi_i)$ with
$\phi_1>\ldots>\phi_n$.

It can be shown that each subcategory $\mc{P} (\phi)$ is
extension-closed and abelian. Its non-zero objects are called
\emph{semistable} of phase $\phi$, while the objects $A_i$ in (d)
are the \emph{semistable factors} of $C$. The minimal objects of
$\kp(\phi)$ are called \emph{stable} of phase $\phi$ (recall that
a \emph{minimal object} in an abelian category, also called
\emph{simple}, is a non-zero object without proper subobjects or
quotients). A HN-filtration of an object $C$ is unique up to a
unique isomorphism.

For any interval $I \subseteq \RR$, $\mc{P} (I)$ is defined to be
the extension-closed subcategory of $\cat{T}$ generated by the
subcategories $\mc{P} (\phi)$, for $\phi \in I$. Bridgeland proved
that, for all $\phi \in \RR$, $\mc{P} ((\phi, \phi + 1])$ is the
heart of a bounded $t$-structure on $\cat{T}$. The category
$\mc{P} ((0, 1])$ is called the \emph{heart} of $\sigma$.

\begin{remark}\label{rmk:tstruct}
By \cite[Prop.\ 5.3]{Br} giving a stability condition on a triangulated category $\cat{T}$ is equivalent to giving a bounded $t$-structure on $\cat{T}$ with heart $\cat{A}$ and a group homomorphism $Z:K(\cat{A})\to\CC$ such that $Z(C)\in\HH$, for all $0\neq C\in\cat{A}$, and with Harder--Narasimhan filtrations (see \cite[Sect.\ 5.2]{Br}). More precisely, as $\HH:=\{ z\in\CC^*:z=|z|\exp(i\pi\phi), \, 0< \phi \leq 1 \}$, any $0\neq C\in\cat{A}$ has a well-defined \emph{phase} $\phi(C):=\mathrm{arg}(Z(C))\in(0,1]$.
For $\phi\in(0,1]$, an object $0\neq C\in\cat{A}$ is then in $\kp(\phi)$ if and only if, for all $C\epi B\neq 0$ in $\cat{A}$, $\phi=\phi(C)\leq\phi(B)$.
\end{remark}

A stability condition is called \emph{locally-finite} if there
exists some $\epsilon > 0$ such that, for all $\phi\in\RR$, each
(quasi-abelian) subcategory  $\mc{P} ((\phi - \epsilon , \phi +
\epsilon))$ is of finite length. In this case $\mc{P} (\phi)$ has
finite length so that every object in $\mc{P} (\phi)$ has a finite
\emph{Jordan--H\"older filtration} (\emph{JH-filtration} for
short) into stable factors of the same phase. The set of stability
conditions which are locally finite will be denoted by $\Stab
(\cat{T})$.

\subsection{The bounded $t$-structure}\label{sec:tstruct}

In this section, we construct a bounded $t$-structure on $\cat{T}_Y$ in a number of steps:
we first consider a tilting by a torsion pair of $\coh(\PP^2,\kb_0)$,
and then restrict it to $\cat{T}_Y$ under the identification of it with an admissible subcategory. In view of the fact that we want the ideal sheaves of lines $I_l$ to be contained in the heart of our $t$-structure, the specific tilt is suggested by
Example \ref{ex:ideallineblownup}. We record the result as the following

\begin{prop}\label{prop:tstruct2}
There exists a bounded $t$-structure on $\cat{T}_Y$ with heart $\cat{B}$ such that

{\rm (i)} $\cat{B}$ has Ext-dimension equal to $2$, i.e., $\Ext^i(C,\widetilde{C})=0$, for all $C,\widetilde{C}\in\cat{B}$ and for all $i\neq0,1,2$;

{\rm (ii)} $I_l\in\cat{B}$, for all lines $l\subseteq Y$.
\end{prop}

\medskip

\noindent
{\bf Step 1.} (Bridgeland's tilting)
Define a torsion pair on $\coh(\PP^2,\kb_0)$ as follows:
\begin{align*}
\kt_0:=& \left\{ K\in\coh(\PP^2,\kb_0)\,:\,\text{either }K\text{ torsion or }\mu^-(K_{tf})>\mu(\kb_0)=-5/4\right\}\\
\kf_0:=& \left\{ K\in\coh(\PP^2,\kb_0)\,:\,K\text{ torsion free and }\mu^+(K)\leq\mu(\kb_0)=-5/4\right\},
\end{align*}
where $K_{tf}$ is defined by Lemma \ref{lem:tfreepart}, and $\mu^+$ (resp.\ $\mu^-$) denotes the biggest (resp.\ smallest) slope of the factors in a Harder--Narasimhan filtration.

We observe here that, by (3) and (4) of Section \ref{sec:mustability}, $\kb_1\in\kt_0$ and $\kb_0\in\kf_0$.
By \cite{HRS} we get a new bounded $t$-structure on $\Db(\PP^2,\kb_0)$. We denote by $\cat{A}_0$ its heart. Explicitly,
\[
\cat{A}_0=\left\{C\in\Db(\PP^2,\kb_0)\,:\,\begin{array}{l}\bullet\ \kh_{\coh}^i(C)=0,\text{ for all }i\neq0,-1\\ \bullet\ \kh_{\coh}^0(C)\in\kt_0\\ \bullet\ \kh_{\coh}^{-1}(C)\in\kf_0\end{array}\right\},
\]
where $\kh_{\coh}^\bullet$ denotes the cohomology with respect to the $t$-structure with heart $\coh(\PP^2,\kb_0)$.

\medskip

\noindent
{\bf Step 2.} (Ext-dimension of $\cat{A}_0$)
An important property of $\cat{A}_0$ is having Ext-dimension equal to 2 as $\coh(\PP^2,\kb_0)$.
Indeed, for all $C,\widetilde{C}\in\cat{A}_0$, we have $\Hom^{<0}(C,\widetilde{C})=0$, by definition of bounded $t$-structure.
Moreover, by construction and Proposition \ref{prop:properties} (ii), $\Hom^{\geq 4}(C,\widetilde{C})=0$.
Hence we only need to show that $\Hom^3(C,\widetilde{C})=0$.
But an easy computation shows that
\[
\Hom^3(C,\widetilde{C})\cong\Hom^2(\kh_{\coh}^{-1}(C),\kh_{\coh}^0(\widetilde{C}))\cong\Hom(\kh_{\coh}^0(\widetilde{C}),\kh_{\coh}^{-1}(C)\otimes_{\kb_0}\kb_{-1})^\vee.
\]
By property (3) of Section \ref{sec:mustability}, $\kh_{\coh}^{-1}(C)\otimes_{\kb_0}\kb_{-1}\in\kf_0$. Hence
\[
\Hom(\kh_{\coh}^0(\widetilde{C}),\kh_{\coh}^{-1}(C)\otimes_{\kb_0}\kb_{-1})=0,
\]
as wanted.

\medskip

\noindent
{\bf Step 3.} ($\kb_1$-dimension of $\cat{A}_0$)
Another good property of $\cat{A}_0$ is that its $\kb_1$-dimension is equal to 1, that is $\Hom^i(\kb_1,\cat{A}_0)=0$ if $i\neq0,1$.
Indeed, by Step 2, we only need to show that $\Hom^2(\kb_1,\cat{A}_0)=0$.
Let $C\in\cat{A}_0$. Then an easy computation, using Step 2, shows that
\[
\Hom^2(\kb_1,C)\cong\Hom^2(\kb_1,\kh_{\coh}^0(C))\cong\Hom(\kh_{\coh}^0(C),\kb_0)=0,
\]
by definition of torsion pair.

\medskip

\noindent
{\bf Step 4.} (Inducing a bounded $t$-structure on $\cat{T}_Y$)\footnote{This step, which simplifies a previous version of our argument, was suggested to us by A.\ Kuznetsov.}
Set $\cat{B}:=(\sigma_*\circ\Phi')(\cat{A}_0)\cap\cat{T}_Y$. By Step $2$, the Ext-dimension of $\cat{B}$ is $2$. The following result concludes the proof of part (i) of Proposition \ref{prop:tstruct2}.

\begin{lem}\label{lem:kuzsug}
	The category $\cat{B}$ is the heart of a bounded $t$-structure on $\cat{T}_Y$.
\end{lem}

\begin{proof}
We prove the claim under the indentification of $\cat{T}_Y$ with its image $\fu(\cat{T}_Y)$ via the functor $\fu:\cat{T}_Y\to\Db(\PP^2,\kb_0)$.
Recall that $\fu(\cat{T}_Y)=\langle \kb_1 \rangle^{\perp}$.
	Consider the spectral sequence (see, for example, \cite{Okada})
	\begin{equation}\label{eq:specseq}
	E_2^{p,q}:=\bigoplus_i \Hom^p(\kh_{0}^i(C),\kh_{0}^{i+q}(\widetilde{C}))\Longrightarrow\Hom^{p+q}(C,\widetilde{C}),
	\end{equation}
	where the cohomology is taken with respect to $\cat{A}_0$ and $C,\widetilde{C}\in\Db(\PP^2,\kb_0)$.
	When $C=\kb_1$, by Step 3, the sequence degenerates at the second order. This means that, if $\widetilde{C}\in\fu(\cat{T}_Y)$,
then $\kh_0^i(\widetilde{C})$ is in $\fu(\cat{T}_Y)$ as well. This is sufficient to prove the
result.
\end{proof}

\medskip

\noindent
{\bf Step 5.} (Ideal sheaves of lines)
Here we will prove that the ideal sheaves of lines of $Y$ belong to $\cat{B}$, thus completing the proof of Proposition \ref{prop:tstruct2}.
A standard calculation shows that $I_l$ is in $\cat{T}_Y$, for any line $l$ in $Y$, and then it is enough to show that $\cd{I_l}$ belongs to $\cat{A}_0$, again for any line $l$ in $Y$. We consider the three possible cases of $l$ and $l_0$ skew, $l=l_0$ and $l$ intersecting $l_0$ in a point.

If $l$ does not intersect $l_0$, then, by Example \ref{ex:idealsheaves}, $\cd{I_l}\cong j_*(\ke|_l)\otimes\ko_{\PP^2}(h)$,
where $j$ denotes the embedding given by the composition $l\mono \widetilde{Y}\mor[\pi]\PP^2$.
This is a torsion sheaf supported on a line in $\PP^2$.
Hence $\cd{I_l}$ belongs to $\kt_0$ and so to $\cat{A}_0$.

If $l=l_0$, then, by Example \ref{ex:ideallineblownup}, $\cd{I_{l_0}}$ is given by the unique extension $\kb_1\to\kb_0[2]$.
In this case, $\cd{I_{l_0}}$ is in $\cat{A}_0$ since $\kb_1\in\kt_0$ and $\kb_0\in\kf_0$.

To prove Proposition \ref{prop:tstruct2}, we show that, if $l\cap l_0=\{pt\}$, then $\cd{I_l}$ is again a torsion sheaf.
Recall that $\cd{I_l}=(\sigma_* \circ \Phi')^{-1}(I_l)$.
A local computation shows that $\sigma^*I_l$ is a sheaf, sitting in an exact sequence
\[
0\to\sigma^*I_l\to\ko_{\widetilde{Y}}\to\ko_{l\cup\gamma}\to0,
\]
where $\gamma:=\sigma^{-1}(pt)$ and $l$ denotes, by abuse of notation, the strict transform of $l$ inside $\widetilde{Y}$.
Now, by \eqref{eq:interestingformula}, we have
\[
\cd{I_l}\cong\Psi(\sigma^*I_l)\cong\Psi(\ko_{l\cup\gamma})[-1]\cong\pi_*(\ke|_{l\cup\gamma}\otimes\ko_{\widetilde{Y}}(h)).
\]
By using the exact sequence
\[
0\to\ko_\gamma(-h)\to\ko_{l\cup\gamma}\to\ko_{l}\to0,
\]
we have an exact triangle
\[
\pi_*(\ke|_{\gamma})\to\Psi(\sigma^*I_l)\to\pi_*(\ke|_{l}),
\]
where $\pi_*(\ke|_{\gamma})$ is a torsion sheaf supported on a line in $\PP^2$, since $\pi$ is a closed embedding on $\gamma$.
By \eqref{eq:defofke}, the sheaf $\ke|_{l}$ has no higher cohomologies and so $\pi_*(\ke|_{l})$ is a torsion sheaf supported
on a point. As a consequence, $\cd{I_l}\cong\Psi(\sigma^*I_l)$ is a torsion sheaf on $\PP^2$, as we wanted.

\subsection{The stability function}\label{subsect:stabfunct}

We now construct a stability condition on $\cat{T}_Y$, using Remark \ref{rmk:tstruct}.
Define a group homomorphism $Z: K(\PP^2,\kb_0)\to\CC$ as follows:
\[
Z([C])=\rk(C)+i(\deg(C)-\mu(\kb_0)\rk(C)),
\]
where $\rk$ and $\deg$ are defined in Section \ref{sec:mustability}. If $A$ is an object of $\cat{T}_Y$,
we define $Z([A]):=Z([\cd{A}])$. In particular we obtain a group homomorphism $Z:K(\cat{T}_Y)\to\CC$ as follows:
\[
Z([A])=\rk(A)+i(\deg(A)-\mu(\kb_0)\rk(A)),
\]
where $\rk(A)=\rk(\cd{A})$ and $\deg(A)= \deg(\cd{A})$ are the two numerical functions of Section \ref{sec:mustability}.

\begin{lem}\label{lem:stabfunction}
The group homomorphism $Z$ has the property $Z(\cat{B}\setminus\{0\})\subseteq\HH$.
\end{lem}

\begin{proof}
By the definition of $\cat{B}$ in Step 4 of the proof of Proposition \ref{prop:tstruct2}, if $C\in\cat{B}$ is non-zero, then its image in $\Db(\PP^2,\kb_0)$ fits into an exact triangle
\[
\kh_{\coh}^{-1}(\cd{C})[1]\to\cd{C}\to\kh_{\coh}^0(\cd{C}),
\]
where the cohomology is taken with respect to $\coh(\PP^2,\kb_0)$.
Then it will be sufficient to prove that $Z(\kh_{\coh}^0(\cd{C}))$ and $Z(\kh_{\coh}^{-1}(\cd{C})[1])$ have the required property.
But $\kh_{\coh}^0(\cd{C})$ is, by definition, in $\kt_0$. Hence
\[
\deg(\kh_{\coh}^0(\cd{C}))-\mu(\kb_0)\rk(\kh_{\coh}^0(\cd{C}))>0.
\]
Indeed, if $\rk(\kh_{\coh}^0(\cd{C}))>0$, then $\mu(\kh_{\coh}^0(\cd{C}))>\mu(\kb_0)$.
By Lemma \ref{lem:rankpoints} (iii), sheaves with torsion supported on points do not belong to
$\fu(\cat{T}_Y)=\langle \kb_1 \rangle^{\perp}$.
Hence, if $\rk(\kh_{\coh}^0(\cd{C}))=0$, then $\deg(\kh_{\coh}^0(\cd{C}))>0$.

Similarly, $\kh_{\coh}^{-1}(\cd{C})\in\kf_0$. Hence
\[
\deg(\kh_{\coh}^{-1}(\cd{C})[1])-\mu(\kb_0)\rk(\kh_{\coh}^{-1}(\cd{C})[1])\geq0.
\]
If equality holds, that is $\mu(\kh_{\coh}^{-1}(\cd{C}))=\mu(\kb_0)$, then $\rk(\kh_{\coh}^{-1}(\cd{C})[1])<0$, as wanted.
\end{proof}

\begin{remark}\label{rem:realslope}
Note that the proof of the lemma above shows in particular that any object $A$ of $\cat{B}$
of real slope is in fact of the form $\tilde{A}[1]$, where $\tilde{A}$ is $\mu$-semistable
sheaf of slope $\mu(\kb_0)$.
\end{remark}

\begin{prop}\label{lem:stability}
The pair $(Z,\cat{B})$ defines a locally finite stability condition $\sigma$ on $\cat{T}_Y$.
\end{prop}

\begin{proof}
We need to show that $\sigma$ has Harder--Narasimhan filtrations and is locally-finite.
But the locally-finiteness condition is obvious, since the image of $Z$ is a discrete subgroup of $\CC$ (see \cite[Sect.\ 5]{Br}).

The existence of Harder--Narasimhan filtrations can be proved using the same ideas as in the proof of \cite[Prop.\ 7.1]{Br1}.
By the criterion of \cite[Prop.\ 2.4]{Br}, this amounts to showing that $\cat{B}$ has no infinite
sequence of subojects with with increasing phase or infinite sequence of quotients with
decreasing phase.

For $G\in\cat{B}$, denote $f(G):=\Im(Z(G))\geq 0$. Clearly $f$ is additive and, if
\[
0\to A\to C\to B\to 0
\]
is an exact sequence in $\cat{B}$, then $f(A)\leq f(C)$ and $f(B)\leq f(C)$. With this in mind, let
\[
\ldots \subseteq C_{j+1}\subseteq C_j\subseteq\ldots\subseteq C_1\subseteq C_0=C
\]
be an infinite sequence in $\cat{B}$ of subobjects of $C$ with $\phi(C_{j+1})>\phi(C_j)$, for all $j$.
Since $f$ is discrete, there exists $N \in \NN$ such that
$0\leq f(C_n)=f(C_{n+1})$, for all $n \geq N$.
Consider the exact sequence in the category $\cat{B}$
$$0\mor C_{n+1}\mor C_n\mor G_{n+1}\mor 0.$$
By the additivity of $f$ we have $f(G_{n+1})=0$, for all $n\geq N$.
But this yields $\phi(G_{n+1})=1$, for all $n\geq N$ and so $\phi(C_{n+1})\leq\phi(C_n)$.
This contradicts our assumptions and so, the property (a) of \cite[Prop.\ 2.4]{Br} is verified.

Now let
\[
C=C_0\twoheadrightarrow C_1\twoheadrightarrow\ldots\twoheadrightarrow C_j\twoheadrightarrow C_{j+1}\twoheadrightarrow\ldots
\]
be an infinite sequence in $\cat{B}$ of quotients of $C$ with $\phi(C_j)>\phi(C_{j+1})$, for all $j$.
As before, $f(C_n)=f(C_{n+1})$, for all $n\geq N$.
Consider the exact sequence in $\cat{B}$
\[
0\to G_n\to C_N\to C_n\to 0,
\]
for $n\geq N$. Then $f(G_n)=0$.

Assume we can produce a short exact sequence
\begin{equation}\label{eqn:seq}
0\to A\to C_N\to B\to 0
\end{equation}
in $\cat{B}$ with $A\in\kp(1)$ and $\Hom(P,B)=0$, for all $P\in\kp(1)$, where
\[
\kp(1):=\left\{M\in\cat{B}\,:\,f(M)=0\right\}.
\]
Then $f(G_n)=0$ yields $G_n\in\kp(1)$, for all $n\geq N$.
Since $\Hom(G_n,B)=0$, $G_n\hookrightarrow C_N$ factorizes through $G_n\hookrightarrow A$.
Hence
\[
0=G_N\subseteq G_{N+1}\subseteq\ldots\subseteq G_n\subseteq\ldots\subseteq A
\]
is an increasing sequence of subobjects of $A$ in $\kp(1)$. By the proof of Lemma \ref{lem:stabfunction}, $\kp(1)$ is of finite length.
Consider the equivalence $\fu(\kp(1))\cong\kp(1)$ induced by $\cat{T}_Y\cong\fu(\cat{T}_Y)$.
The objects of $\fu(\kp(1))$ are shifts of $\mu$-semistable torsion-free sheaves with slope $\mu(\kb_0)$. Thus the previous increasing sequence cannot exist and this proves property (b) of \cite[Prop.\ 2.4]{Br} and so our result.

We only need to prove the existence of the exact sequence \eqref{eqn:seq}, which we do for $M\in\cat{B}$. Then we can write it as
\[
\kh_{\coh}^{-1}(\cd{M})[1]\to\cd{M}\to\kh_{\coh}^0(\cd{M}).
\]
By the description of the objects in $\fu(\kp(1))$ mentioned above,
$$\Hom(\fu(\kp(1)),\cd{M})\cong\Hom(\fu(\kp(1)),\kh_{\coh}^{-1}(\cd{M})[1]).$$
Consider a Harder--Narasimhan filtration of $\kh_{\coh}^{-1}(\cd{M})$ in $\mu$-semistable sheaves.
If we have $\mu^+(\kh_{\coh}^{-1}(\cd{M}))<\mu(\kb_0)$, then $\Hom(\fu(\kp(1)),\cd{M})=0$, and we can take $A=0$ in \eqref{eqn:seq}.
If $\mu^+(\kh_{\coh}^{-1}(\cd{M}))=\mu(\kb_0)$, then we can take for $A$ (or better for $\cd{A}$) the biggest $\mu$-semistable
subobject of $\kh_{\coh}^{-1}(\cd{M})$ in $\coh(\PP^2,\kb_0)$ with slope $\mu(\kb_0)$ and which belongs to $\fu(\cat{T}_Y)$.
Notice that its existence is ensured by the noetherianity of $\coh(\PP^2,\kb_0)$. The claim is proved.
\end{proof}

This concludes the proof of Theorem \ref{thm:tstruct3}. Indeed, the stability condition mentioned in the statement is precisely the pair $(Z,\cat{B})$ studied in Lemma \ref{lem:stability}. Moreover, by Proposition \ref{prop:tstruct2} (ii), the ideal sheaves of lines in $Y$ are contained in $\cat{B}$.

\begin{cor}\label{cor:stab}
If $Y$ is a cubic threefold, then $\Stab(\Db(Y))$ is non-empty.
\end{cor}

\begin{proof}
Take the subcategory $\cat{D}:=\langle\cat{T}_Y,\ko_Y\rangle$ of $\Db(Y)=\langle \cat{T}_Y,\ko_Y,\ko_Y(1)\rangle$.
By Proposition \ref{lem:stability}, there exists a locally finite stability condition $\sigma$ on $\cat{T}_Y$.
By construction, there exists an integer $i$ such that $\Hom^{\leq i}(B,\ko_Y)=0$, for all $B\in\cat{B}$. Define
on the subcategory $\langle\ko_Y[i]\rangle$ a locally finite stability condition such that $\ko_Y[i]$ is in its heart
and has phase $1$. In this way, the hypotheses of \cite[Prop.\ 3.3]{CP} are satisfied and this means that
$\Stab(\cat{D})\neq\emptyset$. Now repeat the same argument for $\Db(Y)=\langle\cat{D},\ko_Y(1)\rangle$.
\end{proof}

\section{Stability of ideal sheaves of lines}\label{sect:stabid}

Let $Y$ be a cubic threefold and consider the stability condition $\sigma=(Z,\cat{B})$ on $\cat{T}_Y$ defined above (see Proposition \ref{lem:stability}). This section is devoted to the proof of the following result.

\begin{thm}\label{thm:main}
All ideal sheaves of lines in $Y$ are $\sigma$-stable and they are the only $\sigma$-stable objects in $\cat{B}$ in their numerical class.
\end{thm}

As the proof of this result is somewhat involved and will be the subject of the whole section, we briefly sketch our argument here. In Section \ref{subsec:stab-ideal} we prove that they are stable in the above mentioned stability condition. This follows from a more general result (see Proposition \ref{prop:Mukai}) saying that an object $C$ of $\cat{T}_Y$ is stable in this stability condition as soon as $\Hom(C,C)$ is $2$-dimensional.

We finally show in Section \ref{subsec:4.3} that the ideal sheaves of lines are the only stable objects in $\cat{B}$ in their numerical class. This is the result of a long calculation boiling down to showing that if there is a stable object $A\in\cat{B}$ with the same numerical class as the ideal sheaf $I_l$ of a line, then $A$ (or rather its image via an appropriate functor) is torsion-free and $\ch(A)=\ch(I_l)$.

\subsection{The ideal sheaves of lines are $\sigma$-stable}\label{subsec:stab-ideal}

In this section we prove the stability of special objects of $\cat{T}_Y$.
In particular this implies that for any line $l$ in $Y$, the ideal sheaf $I_l$ is $\sigma$-stable.

\begin{prop}\label{prop:Mukai}
Let $C\in\cat{T}_Y$ be such that $\Hom^1(C,C)\cong\CC^2$.
Then $C$ is $\sigma$-stable.
\end{prop}

\begin{cor}\label{cor:stabid}
Let $l\subseteq Y$ be any line. Then $I_l$ is $\sigma$-stable.
\end{cor}

The section is entirely dedicated to the proof of Proposition \ref{prop:Mukai}.
Let us start with two lemmas.

\begin{lem}\label{lem:nonexistence}
There exists no non-zero object $C\in\cat{B}$ with either $\Hom^1(C,C)=0$ or $\Hom^1(C,C)\cong\CC$.
\end{lem}

\begin{proof}
This follows easily from Proposition \ref{prop:GrotGroup} (ii): indeed, $\chi(C,C)\leq-1$.
Hence, since the Ext-dimension of $\cat{B}$ is equal to 2, if either $\Hom^1(C,C)=0$ or $\Hom^1(C,C)\cong\CC$, then $\chi(C,C)\geq0$, a contradiction.
\end{proof}

\begin{lem}\label{lem:Mukai}
Let $C\in\cat{T}_Y$ be such that $\Hom^1(C,C)\cong\CC^2$.
Then, up to shift, $C$ belongs to $\cat{B}$.
\end{lem}

\begin{proof}
Consider again the spectral sequence \eqref{eq:specseq}
\begin{equation}\label{eqn:sp1000}
E_2^{p,q}:=\bigoplus_i \Hom^p(\kh_{\cat{B}}^i(C),\kh_{\cat{B}}^{i+q}(C))\Longrightarrow\Hom^{p+q}(C,C),
\end{equation}
where the cohomology is taken with respect to $\cat{B}$.
Since the Ext-dimension of $\cat{B}$ is equal to $2$ (see Proposition \ref{prop:tstruct2}), then the $E_2^{1,q}$ terms of the spectral sequence survive.
In particular, for $q=0$ we have
\[
2=\hom^1(C,C)\geq\sum_i\hom^1(\kh_{\cat{B}}^i(C),\kh_{\cat{B}}^i(C)).
\]
But, by Lemma \ref{lem:nonexistence},
\[
\sum_i\hom^1(\kh_{\cat{B}}^i(C),\kh_{\cat{B}}^i(C))\geq2r,
\]
where $r\geq1$ is the number of non-zero cohomologies of $C$.
Hence $r=1$, as wanted.
\end{proof}

By Lemma \ref{lem:Mukai}, we can assume that any $C$ as in Proposition \ref{prop:Mukai} is in $\cat{B}$. Since by Proposition \ref{prop:tstruct2} (i) the $\Ext$-dimension of $\cat{B}$ is $2$ and $\chi(C,C)\leq -1$, then necessarily $\chi(C,C)=-1$. By Lemma \ref{prop:-1classes}, $[C]$ is either $[I_l]$, or $[S_{\cat{T}_Y}(I_l)]$, or $-[S^2_{\cat{T}_Y}(I_l)]$. In particular, the class of $[C]$ is primitive and to prove Proposition \ref{prop:Mukai} it suffices to show that $C$ is $\sigma$-semistable.

If $[C]=-[S^2_{\cat{T}_Y}(I_l)]$, then $Z(C)=Z(I_l)-Z(S_{\cat{T}_Y}(I_l))\in\RR_{<0}$ and so it is $\sigma$-semistable. Otherwise, $\Im(Z(C))=2$ and the $\sigma$-semistability follows from the lemma below.

\begin{lem}\label{lem:}
	Let $C\in\cat{B}$ be such that $\Hom^1(C,C)\cong\CC^2$ and $\Im(Z(C))=2$.
	Then $C$ is $\sigma$-semistable.
\end{lem}

\begin{proof}
Assume $C$ is not $\sigma$-semistable.
Take $0\to A\to C\to B\to0$, a destabilizing sequence in $\cat{B}$, where $B$ is $\sigma$-semistable.
Then $\Im(Z(A))+\Im(Z(B))=2$. But, by Lemma \ref{lem:rankpoints} (ii), $\Im(Z)$ is always even and positive.
Hence $\Im(Z(B))=2$ and $\Im(Z(A))=0$, so that $A$ is $\sigma$-semistable, and, by Remark
\ref{rem:realslope}, $A\cong\widetilde{A}[1]$, a shift by $1$ of a $\mu$-semistable sheaf
$\widetilde{A}$ with slope equal to $\mu(\kb_0)$.
We claim that $\Hom^{\geq 2}(B,A)=0$.
Indeed, since $A$ and $B$ belong to $\cat{B}$, we only need to examine $\Hom^2(B,A)$.
An easy computation shows that
\[
\Hom^2(B,A)\cong\Hom^2(\kh_{\coh}^{-1}(\cd{B}),\cd{\widetilde{A}})\cong\Hom(\cd{\widetilde{A}},\kh_{\coh}^{-1}(\cd{B})\otimes_{\kb_0}\kb_{-1})^{\vee},
\]
where $\kh_{\coh}^\bullet$ denotes the cohomology sheaves of $\cd{B}$ with respect to $\coh(\PP^2,\kb_0)$.
But then, by Lemma \ref{lemma:property(3)}, $\mu^+(\kh_{\coh}^{-1}(\cd{B})\otimes_{\kb_0}\kb_{-1})<\mu(\kb_0)=\mu(\cd{\widetilde{A}})$.
Hence, $\Hom(\cd{\widetilde{A}},\kh_{\coh}^{-1}(\cd{B})\otimes_{\kb_0}\kb_{-1})=0$, as wanted.

Write $\phi:=\phi(C)$, and
consider the abelian category $\cat{B}_{\phi}:= \kp ((\phi-1,\phi])$ in $\cat{T}_Y$. Denote by
$\kh_{\phi}^\bullet$ the cohomology with respect to the corresponding $t$-structure. Then we have a triangle
\[
\kh_{\phi}^{-1}(C)[1]\to C\to\kh_{\phi}^0(C),
\]
where $\kh_{\phi}^{-1}(C)=\widetilde{A}$ and $\kh_{\phi}^0(C)=B$.
By using once more the spectral sequence \eqref{eqn:sp1000} (with respect to the heart $\cat{B}_{\phi}$), the vanishing of $\Hom^2(B,A)\cong\Hom^3(B,\widetilde{A})$ yields that the $E_2^{1,0}$-term survives.
Hence
\[
\hom^1(A,A)+\hom^1(B,B)\leq2.
\]
By Lemma \ref{lem:nonexistence} this is impossible and so $C$ is $\sigma$-semistable.
\end{proof}

\subsection{Ideal sheaves of lines as unique stable objects in their numerical class}\label{subsec:4.3}

In this section, we treat the final step  to conclude the proof of Theorem \ref{thm:main}.
At this point, we only need to show that the ideal sheaves $I_l$ are the only
$\sigma$-stable objects in $\cat{B}$ in their numerical class.

Let $M \in \cat{B}$ be a $\sigma$-stable object with numerical class $[\cd{I_l}]$.
We now prove that $\fu^{-1}(M)=(\sigma_*\circ\Phi')(M)$ is a torsion-free sheaf on $Y$ with Chern character equal to $\ch(I_l)$. As we will observe at the end of this section, this is enough to conclude since the Picard group of $Y$ is isomorphic to $\ZZ$. This means that if a torsion-free sheaf $A$
on $Y$ satisfies $\ch(A)=\ch(I_l)$, then there exists a line $m$ on $Y$ such that $A\cong I_m$.

First notice that, by Example \ref{ex:ideallineblownup}, $[M]=[\kb_1]-[\kb_0]$, and, in
particular, $\rk(M)=0$ and $\deg(M)=2$.

\begin{lem}\label{lem:Lemma1}
We have either $M\cong\cd{I_{l_0}}$ or $M\in\coh(\PP^2,\kb_0)$.
\end{lem}

\begin{proof}
Assume that $M\notin\coh(\PP^2,\kb_0)$, and so that $\kh_{\coh}^{-1}(M)\neq0$.
Then, since $\rk(M)=0$ and $\rk(\kh_{\coh}^{-1}(M))\neq0$, we have $\kh_{\coh}^0(M)\neq0$, too, and an exact triangle
\begin{equation}\label{eq:1717}
\kh_{\coh}^{-1}(M)[1]\to M\to\kh_{\coh}^0(M).
\end{equation}
We divide the proof in three steps.

\medskip

\noindent
{\bf Step 1.} We have that $\kh_{\coh}^0(M),\kh_{\coh}^{-1}(M)\notin\fu(\cat{T}_Y)$.
Indeed, suppose the contrary, namely that either $\kh_{\coh}^0(M)$ or $\kh_{\coh}^{-1}(M)$ is in
$\fu(\cat{T}_Y)$.
Then both the two cohomologies are in $\fu(\cat{T}_Y)$ (and so in $\fu(\cat{B})$), and
since $\rk(\kh_{\coh}^0(M))>0$, $\phi(\kh_{\coh}^{0}(M))<1/2=\phi(M)$. But then
\eqref{eq:1717} destabilizes $M$, a contradiction.

\medskip

\noindent
{\bf Step 2.} Here we want to prove that $\mu^-(\kh_{\coh}^0(M)_{tf})\geq\mu(\kb_1)$.
By Step 1, $M_0:=\kh_{\coh}^0(M)$ does not belong to $\fu(\cat{T}_Y)$.
Hence, by Proposition \ref{prop:properties} (i),  $\Hom^*(\kb_1,M_0)\neq0$;
we claim that  $\Hom^p(\kb_1,M_0)=0$, for all $p\neq0$, so that
$\Hom^*(\kb_1,M_0)\cong\CC^{\oplus a_0}$, for some $a_0>0$. Indeed, $\Hom^p(\kb_1,M_0)=0$ if $p<0$.
To see that the same holds for $p>0$, observe that applying the functor $\mathbf{R}\Hom(\kb_1,-)$ to \eqref{eq:1717}, we get the exact sequence
\[
\Hom^p(\kb_1,M)\longrightarrow\Hom^p(\kb_1,M_0)\longrightarrow\Hom^{p+2}(\kb_1,\kh_{\coh}^{-1}(M)).
\]
The first vector space is trivial for any $p$ because $M\in\fu(\cat{T}_Y)$ (\ref{prop:properties}). The last one is trivial for all $p>0$ because $\coh(\PP^2,\kb_0)$ has $\mathrm{Ext}$-dimension $2$.

Consider the evaluation map $ev_0:\kb_1^{\oplus a_0}\to M_0$.
If it is surjective, since $\kb_0$ is $\mu$-stable, we have $\mu^-((M_0)_{tf})\geq\mu(\kb_1)$, as wanted.

Assume $ev_0$ is not surjective.
Set $M_1:=\coker(ev_0)$.
Then, $\mathrm{cone}(ev_0)\in \langle\kb_1\rangle^{\perp}=\fu(\cat{T}_Y)$ and we have an exact triangle
\[
\ker(ev_0)[1]\to\mathrm{cone}(ev_0)\to M_1.
\]
As before, if $\Hom^p(\kb_1,M_1)\neq0$, then $p=0$.

Assume $\Hom(\kb_1,M_1)=0$. Then $M_1\in\fu(\cat{T}_Y)$.
Moreover, $M_0\epi M_1$ in $\coh(\PP^2,\kb_0)$ implies $M_1\in\kt_0$ and so $M_1\in\fu(\cat{B})$.
Consider the composition
\[
\psi:M\to M_0\to M_1.
\]
Then $\psi\neq0$ and the image $\im(\psi)$ of $\psi$ in the abelian category $\fu(\cat{B})$ is non-trivial.
Since $\im(\psi)\mono M_1$, where $M_1$ is a sheaf,  $\im(\psi)\in\coh(\PP^2,\kb_0)$.
So, we found a surjection $M\epi\im(\psi)$ in $\fu(\cat{B})$ with $1/2=\phi(M)\geq\phi(\im(\psi))$, contradicting the $\sigma$-stability of $M$.

Hence, $\Hom(\kb_1,M_1)\cong\CC^{\oplus a_1}\neq0$.
Proceeding as before, we consider the evaluation map $ev_1:\kb_1^{\oplus a_1}\to M_1$, set $M_2:=\coker(ev_1)$, and so on.
What we have produced is a sequence of quotients
\[
M_0\epi M_1\epi M_2\epi\ldots
\]
in $\coh(\PP^2,\kb_0)$. Since this is a noetherian abelian category, there exists some $m\gg0$
with $M_m=0$.

Now we can conclude with an inductive argument in a finite number of steps. Let $s$ be the smallest integer such that $\rk(M_{s+1})=0$. For any $j=0,\ldots s$, we have a map of short
exact sequences:
\begin{equation*}
\xymatrix{
0 \ar[r] & \ar[d] (M_{j-1})_{tor} \ar[r] & \ar[d] M_{j-1} \ar[r] & \ar[d] (M_{j-1})_{tf} \ar[r] & 0\\
0 \ar[r] &  (M_{j})_{tor} \ar[r] &  M_{j} \ar[r] &  (M_{j})_{tf} \ar[r] & 0}
\end{equation*}
The nine-lemma applied to this yields a short exact sequence in $\coh(\PP^2,\kb_0)$
\[
0\to Q\to (M_{j-1})_{tf}\to (M_j)_{tf}\to 0,
\]
where $Q$ sits in a short exact sequence in $\coh(\PP^2,\kb_0)$
\[
0\to A'\to Q\to B'\to 0,
\]
where $\kb_1^{\oplus a_{j-1}}\epi A'$ and $\rk(B')=0$. Therefore, $\mu^-(Q)\geq\mu(\kb_1)$ and, to deduce that $\mu^-((M_{j-1})_{tf})\geq\mu(\kb_1)$, it is enough to show that the same holds for $\mu^-((M_{j})_{tf})\geq\mu(\kb_1)$.

Thus we have reduced ourselves to treating the case $j=s$. By our assumptions, $(M_s)_{tf}$ sits in a short exact sequence in $\coh(\PP^2,\kb_0)$
\[
0\to A\to (M_s)_{tf}\to B\to 0,
\]
where $\kb_1^{\oplus a_s}\epi A$ and $\rk(B)=0$. Again this implies $\mu^-((M_s)_{tf})\geq\mu(\kb_1)$, as wanted.

\medskip

\noindent
{\bf Step 3.} Set $r:= \rk(\kh_{\coh}^0(M))=\rk(\kh_{\coh}^{-1}(M))\neq0$.
By Step 2, we have
\begin{equation*}
\begin{split}
2=\deg(\kb_1)-\deg(\kb_0)&=\deg(M)=\deg(\kh_{\coh}^0(M))-\deg(\kh_{\coh}^{-1}(M))\\
    &=\deg(\kh_{\coh}^0(M)_{tor})+\deg(\kh_{\coh}^0(M)_{tf})-\deg(\kh_{\coh}^{-1}(M))\\
    &\geq\deg(\kh_{\coh}^0(M)_{tf})-\deg(\kh_{\coh}^{-1}(M))\\
    &\geq \frac{r}{4}(\deg(\kb_1)-\deg(\kb_0)).
\end{split}
\end{equation*}
Thus, $r\leq4$. But then,  by Lemma \ref{lem:rankpoints} (i), $r=4$, and the inequalities are equalities. In particular, $\deg(\kh_{\coh}^0(M)_{tor})=0$, $\mu(\kh_{\coh}^0(M)_{tf})=\mu(\kb_1)$,
and  $\mu(\kh_{\coh}^{-1}(M))=\mu(\kb_0)$. It follows that
$\kh_{\coh}^0(M)$ has torsion only on points, its torsion-free part is $\mu$-stable, and
$\kh_{\coh}^{-1}(M)$ is $\mu$-stable too.
But, by Step 1, $\Hom(\kb_1,\kh_{\coh}^0(M))\cong\Hom^2(\kb_1,\kh_{\coh}^{-1}(M))\neq0$.

Consider a non-zero morphism $\kb_1\to\kh_{\coh}^0(M)$.
Then, by stability, either $\kb_1\to\kh_{\coh}^0(M)_{tor}$ or $\kb_1\mono\kh_{\coh}^0(M)_{tf}$.
But, since $\kh_{\coh}^0(M)_{tor}$ is supported on points, if all homomorphisms $\kb_1\to\kh_{\coh}^0(M)$ factorize through $\kh_{\coh}^0(M)_{tor}$, then, by Lemma \ref{lem:rankpoints} (iii), $\kh_{\coh}^0(M)_{tf}$ is in $\fu(\cat{T}_Y)$ and so it destabilizes $M$. Thus $\kb_1\mono\kh_{\coh}^0(M)_{tf}$.

In the same way, by Serre duality, $\kh_{\coh}^{-1}(M)\mono\kb_0$.
Set $T_1:=\kh_{\coh}^0(M)_{tf}/\kb_1$ and $T_2:=\kb_0/\kh_{\coh}^{-1}(M)$.
Then we have
\begin{equation*}
\begin{split}
[\kb_1]-[\kb_0]&=[M]=[\kh_{\coh}^0(M)]-[\kh_{\coh}^{-1}(M)]\\
       &=[\kh_{\coh}^0(M)_{tor}]+[\kb_1]+[T_1]-[\kb_0]+[T_2],
\end{split}
\end{equation*}
and so $[\kh_{\coh}^0(M)_{tor}]+[T_1]+[T_2]=0$. But then, since these sheaves are supported
on points,
\[
\kh_{\coh}^0(M)_{tor}=T_1=T_2=0,
\]
and we obtain $\kh_{\coh}^0(M)\cong\kb_1$ and $\kh_{\coh}^{-1}(M)\cong\kb_0$, that is $M\cong\cd{I_{l_0}}$, as required.
\end{proof}

By the previous lemma, we can assume $M\in\coh(\PP^2,\kb_0)$.
Consider $\Phi(M)\in\Db(\widetilde{Y})$.
Since $\ke'$ is a flat left $\pi^*\kb_0$-module, $\Phi(M)$ is a sheaf.

\begin{lem}\label{lem:vanishing}
We have
\begin{equation*}
\Ext^p(\Phi(M),\ko_{\widetilde{Y}}(-h))=\begin{cases}\CC,&\mathrm{if}\ p=1\\ 0,&\mathrm{otherwise.}\end{cases}
\end{equation*}
\end{lem}

\begin{proof}
First of all, by the adjunction $\Phi\dashv \Pi$ (\ref{bascoh}),
\[
\Ext^p(\Phi(M),\ko_{\widetilde{Y}}(-h))\cong\Ext^p(M,\Pi(\ko_{\widetilde{Y}}(-h))).
\]
We have that $\Pi(\ko_{\widetilde{Y}}(-h))\cong\kb_1$. Indeed, since the normal bundle of $\widetilde{Y}$ in
$\widetilde{\PP}^4$ is $\ko_{\widetilde{Y}}(2H+h)$, we get
\begin{equation*}
\begin{split}
\Pi(\ko_{\widetilde{Y}}(-h))&= \pi_* \kh om(\ke', \ko_{\widetilde{Y}}(-h))\\
&\cong q_* \alpha_* \kh om(\ke' (2H+h), \alpha^! \ko_{\widetilde{\PP}^4}(-h)[1])\\
&\cong q_* \kh om(\alpha_*(\ke') (2H+h), \ko_{\widetilde{\PP}^4}(-h)[1]),
\end{split}
\end{equation*}
where for the last isomorphism we use relative Serre duality. Since $\alpha_* \ke'$ sits in the short exact sequence \eqref{eq:defofkeprime}, we obtain the isomorphisms
$$\Pi(\ko_{\widetilde{Y}}(-h))\cong\kb_0^{\vee}(-2h)\cong\kb_1.$$

Since $M$ is a torsion sheaf, $\Ext^p(M,\kb_1)=0$, for all $p\neq1,2$.
Moreover,
\[
\chi(M,\kb_1)=\chi(\kb_0[1],\kb_1)+\chi(\kb_1,\kb_1)=-1.
\]
Hence, we only need to show that $\Ext^2(M,\kb_1)=0$.

By Proposition \ref{prop:properties}, $\Ext^2(M,\kb_1)\cong\Hom(\kb_0(h),M)^{\vee}$.
Assume that $\Hom(\kb_0(h),M)\neq0$. One calculates that $\hom(\kb_1,\kb_0(h))>1$
using (\ref{bcoh}), and the free-forgetful adjunction. Thus, there is
an exact sequence in $\coh(\PP^2,\kb_0)$
\begin{equation}\label{eq:exseqmult}
0\to\kb_1\to\kb_0(h)\to Q\to0,
\end{equation}
where, for reasons of homological dimension, $Q:=\kb_0(h)/\kb_1$ is pure of dimension $1$.
Since $M\in\cat{T}_Y$, the composition $\kb_1\to\kb_0(h)\to M$ is the zero map.
Hence we have a non-zero morphism $\eta:Q\to M$. Suppose $\eta$ is not injective. Then, by Lemma \ref{lem:rankpoints} (ii), $\ker(\eta)$ has degree $2$ (as $\rk (\ker (\eta))= 0$ and $\deg(Q)=2$). By Lemma \ref{lem:notorsion0} below, $M$ is also pure of dimension $1$ and so $\ker(\eta)=0$.
By applying $\Hom(\kb_1,-)$ to the exact sequence \eqref{eq:exseqmult}, we have $\Hom(\kb_1,Q)\neq0$ and then $\Hom(\kb_1,M)\neq0$, a contradiction.
\end{proof}

\begin{lem}\label{lem:notorsion0}
$\mathrm{Forg}(M)$ is a pure sheaf on $\PP^2$ of dimension $1$.
\end{lem}

\begin{proof}
Denote by $T$ the torsion part of $\mathrm{Forg}(M)$ supported in dimension $0$.
It is easy to see, as in Lemma \ref{lem:tfreepart}, that then $T$ has a structure of $\kb_0$-module for which $T\mono M$ in $\coh(\PP^2,\kb_0)$.
But then, by Lemma \ref{lem:rankpoints} (iii), $\Hom(\kb_1,T)\neq0$ unless $T=0$. So $\Hom(\kb_1,M)\neq0$, contradicting $M\in\fu(\cat{T}_Y)$.
\end{proof}

Since $\Phi'=R_{\ko_{\widetilde{Y}}(-h)}\circ\Phi$, by Lemma \ref{lem:vanishing}, $\Phi'(M)$ is given by an extension
\begin{equation}\label{eqn:z}
0\to\ko_{\widetilde{Y}}(-h)\to\Phi'(M)\to\Phi(M)\to0.
\end{equation}
In particular, it is a sheaf on $\widetilde{Y}$. Moreover, $\ch(\Phi'(M))=\ch(\sigma^*(I_l))=(1,0,*,*)$ and so $\ch(\Phi(M))=(0,h,*,*)$.

\begin{lem}\label{lem:torsionfreeM}
$\Phi'(M)$ is torsion-free.
\end{lem}

\begin{proof}
Assume, for a contradiction, that $\Phi'(M)$ is not torsion-free.
Then there exists an exact sequence
\begin{equation}\label{eqn:z1}
0\to T\to\Phi'(M)\to N\to0,
\end{equation}
with $T$ (resp.\ $N$) a torsion (resp.\ torsion-free) sheaf on $\widetilde{Y}$.
But then $T\mono\Phi(M)$ by \eqref{eqn:z}. Hence, either $T$ has dimension $\leq 1$ or $\ch(T)=(0,h,*,*)$.
The first case is untenable because, by Lemma \ref{lem:notorsion0}, $\pi^*(M)$ has no torsion in dimension $\leq 1$, and, consequently, neither has $\Phi(M)=\pi^*(M)\otimes_{\pi^*\kb_0}\ke'$.

In the second case $\Phi(M)/T$ has dimension $\leq1$ and, by the Snake Lemma (applied to the diagram obtained from \eqref{eqn:z} and \eqref{eqn:z1}), we get an extension
\[
0\to\ko_{\widetilde{Y}}(-h)\to N\to\Phi(M)/T\to0,
\]
which must be trivial by Serre duality. Since $N$ is torsion-free, this is a contradiction unless $T\cong\Phi(M)$.
But this conclusion leads to a contradiction also, since then the sequence (\ref{eqn:z})
splits by Serre duality, and
$\Phi'(M)\cong\Phi(M)\oplus\ko_{\widetilde{Y}}(-h)$ cannot be $\sigma$-stable.
\end{proof}

As $M\in \cat{B}$, there exits an object $\Gamma\in\Db(Y)$, such that
$\Phi'(M)\cong\sigma^*\Gamma$; by the projection formula, $\Gamma=\sigma_*(\Phi'(M))$.
We claim that $\Gamma$ is a torsion-free sheaf on $Y$.
Indeed, the fact that it is a sheaf follows easily since, if not, then there exists an exact triangle
\[
C_0\to\Gamma\to C_1[-1]
\]
in $\Db(Y)$ with $C_0,C_1\in\coh(Y)$.
But then $\kh^0(\sigma^*(C_1))=0$ (where the cohomology is taken in $\coh(\widetilde{Y})$), a contradiction unless $C_1=0$.
Since $\sigma$ is surjective, $\Gamma$ is torsion-free.

Summing up, we have shown that $(\sigma_*\circ\Phi')(M)$ is a torsion-free sheaf on $Y$ with Chern character equal to $\ch(I_l)$.
But, since the Picard group of $Y$ is isomorphic to $\ZZ$, this implies that $(\sigma_*\circ\Phi')(M)$ is the ideal sheaf of some line in $Y$. This proves that if an object in $\cat{B}$ is $\sigma$-stable and in the same numerical class of an ideal sheaf of a line in $Y$, then it is actually isomorphic to an ideal sheaf of a line. Together with Corollary \ref{cor:stabid}, this completes the proof of Theorem \ref{thm:main}.

\section{Proof of Theorem \ref{conj:Kuzne3folds}}\label{sect:finalproof}

In this section we prove the main result of this paper as an application of the results in the previous sections. The strategy is to show that the ideal sheaves of lines on a cubic threefold are preserved by the action of any equivalence (up to composing with a suitable power of the Serre functor, followed by a shift). We complete the proof of Theorem \ref{conj:Kuzne3folds} in Section \ref{subsect:isomor}, by showing that any such equivalence induces an isomorphism between the Fano surfaces of lines.

\subsection{Ideal sheaves go to ideal sheaves}\label{subsect:Fanoprelim}

We begin with a simple consequence of the results of the previous section showing how, given two smooth and projective cubic threefolds $Y$ and $Y'$ and an exact equivalence $U:\cat{T}_{Y'}\isomor\cat{T}_Y$ one can produce another exact equivalence inducing a bijection between the ideal sheaves of lines.

\medskip

Adopting our earlier notation, take $I_{l'}$ an ideal sheaf of a line $l'$ in $Y'$ and consider it as an object in $\cat{T}_{Y'}$. Then consider the object $U(I_{l'})\in\cat{T}_Y$.
Since $U$ is an equivalence of categories, it induces and isomorphism on the Grothendieck groups sending $[I_{l'}]$ to $[U(I_{l'})]$ and preserving the pairing $\chi$. Hence the numerical class $c$ of $U(I_{l'})$ satisfies $\chi(c,c)=-1$.
By Lemma \ref{prop:-1classes}, up to composing with some power of the Serre functor of $\cat{T}_Y$, we can assume $c=[I_l]$, where $l$ is a line in $Y$. Moreover, by Lemma \ref{lem:Mukai}, there is an integer $n_{l'}$ such that $U(I_{l'})[n_{l'}]\in\cat{B}$, where $\cat{B}$ is the heart constructed in Theorem \ref{thm:tstruct3}. Let $\sigma=(Z,\kp)$ be the stability condition constructed in Section \ref{subsect:stabfunct}. Then, by Proposition \ref{prop:Mukai}, $U(I_{l'})[n_{l'}]$ is $\sigma$-stable.

We now want to prove that the shift above can be chosen uniformly. Indeed, given two lines $l',m'$ in $Y'$, assume $U(I_{l'})[n_{l'}]$ is $\sigma$-stable with phase $1/2$ and $U(I_{m'})[n_{m'}]$ is $\sigma$-stable of phase $\phi\in\RR$.
Then $\Hom(I_{m'},I_{l'}[1]),\Hom(I_{l'},I_{m'}[1])\neq0$ and property (c) in the definition of a stability condition give the bound $-1/2<\phi<3/2$.
But then $\phi=1/2$ and so $n_{l'}=n_{m'}=n$.
By Theorem \ref{thm:main}, there exist two lines $l,m$ in $Y$ such that $U(I_{l'})[n]\cong I_l$ and $U(I_{m'})[n]\cong I_m$ in $\cat{T}_Y$. In summary, we proved the following result.

\begin{prop}\label{prop:idealsheavesandequiv}
Let $Y$ and $Y'$ be two cubic threefolds such that $\cat{T}_{Y'}\cong\cat{T}_{Y}$.
Then there exists an equivalence $U:\cat{T}_{Y'}\isomor\cat{T}_{Y}$ which maps ideal sheaves of lines in $Y'$ bijectively onto ideal sheaves of lines in $Y$.
\end{prop}

\subsection{Universal families of lines and convolutions}\label{subsect:convuniv}


\medskip

Recall that the classical Torelli Theorem for a cubic threefold $Y$ says that
the isomorphism class of $Y$ is characterized by its intermediate Jacobian $J(Y)$, thought of
as a principally polarized abelian variety (see \cite{CG}). Furthermore, this latter invariant
itself can be recovered from the Fano variety of lines $F(Y)$. Precisely: the intermediate Jacobian $J(Y)$ is isomorphic to the Albanese variety of $F(Y)$, while the natural
polarization on $J(Y)$ is the class of the image of $F(Y)\times F(Y)$ in $J(Y)$ via the
map $(s,t)\mapsto s-t$. Thus, the natural approach to proving
Theorem \ref{conj:Kuzne3folds} is by showing that the bijection of ideal sheaves of
Proposition \ref{prop:idealsheavesandequiv} in fact induces an isomorphism
$F(Y)\cong F(Y')$.

We hope to produce this isomorphism by appealing to the well known fact the the Fano surface
$F(Y)$ is isomorphic to moduli space of ideal sheaves of lines. Following \cite{Inaba,Lieb}, define the functor
\[
Fano_{Y}:(\mathrm{Sch}/\CC)\to(\mathrm{Set})
\]
by sending a $\CC$-scheme $S$ to the set of equivalence classes of relatively perfect complexes $\ki\in\Db(Y\times S)$ (cf.\ \cite[Def.\ 2.1.1 and Cor.\ 4.3.4]{Lieb}) such that, for all $s\in S$, $\ki|_{Y\times s}$ is an ideal sheaf of a line in $Y$. Then $Fano_Y$ is represented by
$F(Y)$, which is a smooth projective surface (for a recollection about these and other properties of $F(Y)$ see, for example, the introduction of \cite{CG}). Write $\ki'$ for the universal ideal sheaf on
$(Y'\times F(Y'))$.

Define the composite functor
\[
R:\Db(Y')\mor[\rho']\cat{T}_{Y'}\mor[U]\cat{T}_Y\mor[\epsilon]\Db(Y),
\]
where $\epsilon$ is the embedding of $\cat{T}_Y$ into $\Db(Y)$, $\rho'$ is the natural projection from $\Db(Y')$ to $\cat{T}_{Y'}$ and  $U:\cat{T}_{Y'}\isomor\cat{T}_Y$ is an exact equivalence as in the previous section. If $R$ were a \emph{Fourier--Mukai functor}, i.e.\ $R\cong\Phi_\kg$ for some $\kg\in\Db(Y'\times Y)$, the object
\[
\tilde\ki:=\Phi_\kg\times\id_{F(Y')}(\ki')\in \Db(Y\times F(Y')),
\]
would be a family of ideal sheaves of $Y$ parametrized by $F(Y')$,
where
$$\Phi_\kg\times\id_{F(Y')}:=\Phi_{\kg\boxtimes\ko_{\Delta_{F(Y')}}}:\Db(Y'\times F(Y'))\to\Db(Y\times F(Y')).$$
By the universal property of $F(Y)$, this would yield a morphism
$F(Y)\to F(Y')$, which, as in Section \ref{subsect:isomor}, could be verified to be an
isomorphism.

Unfortunately, although $R$ is conjecturally expected to be of Fourier--Mukai type (see \cite[Conj.\ 3.7]{KuzHPD}), we are unable to prove this at the moment. Thus we are forced to build
our ``family of ideal sheaves'' $\tilde\ki$ by hand. Our strategy is to follow Orlov's
method (as reworked in \cite{CS}) of constructing the kernel of an embedding of derived categories by means of convolutions. 

\medskip

As a first step, let $L\in\Pic(F(Y'))$ be an ample line bundle and consider the (infinite) resolution of $\ki'$
\begin{equation}\label{eqn:diag}
\cdots\mor \ko_{Y'}(-N_i)^{\oplus n_i}\boxtimes (L^{-r_i})^{\oplus s_i}\mor\cdots\mor \ko_{Y'}(-N_0)^{\oplus n_0}\boxtimes(L^{-r_0})^{\oplus s_0}\mor\ki'\mor0,
\end{equation}
with the assumptions $r_i\in\NN$ and $N_i\gg0$ so that, if $p\neq 3$,
$$\Hom(\ko_{Y'},\ko_{Y'}(-N_i)[p])=\Hom(\ko_{Y'}(1),\ko_{Y'}(-N_i)[p])= 0.$$

Choose $m$ sufficiently large and truncate \eqref{eqn:diag}, getting a bounded complex
\begin{equation}\label{eqn:diag1}
O_m^\bullet:=\{\ko_{Y'}(-N_m)^{\oplus n_m}\boxtimes (L^{-r_m})^{\oplus s_m}\mor\cdots\mor \ko_{Y'}(-N_0)^{\oplus n_0}\boxtimes(L^{-r_0})^{\oplus s_0}\}.
\end{equation}
Let  $K_m:=\ker(O^m_m\to O_m^{m-1})\in\coh(Y'\times F(Y'))$; then the exact triangle
\[
K_m[m]\mor O_m^\bullet\mor\ki'
\]
splits (if $m>\dim{Y'\times F(Y')}$) . Hence $O^\bullet_m$ has a right
convolution $\ki'\oplus K_m[m]$, which is unique up to isomorphism as
the terms of  $O^\bullet_m$ are coherent sheaves (rather than general
objects in $\Db(Y'\times F(Y'))$; see \cite{Ka}, Lemma 2.1).

For the convenience of the reader, let us recall that a \emph{right convolution} of a bounded complex
\[
A_{m}\mor[d_{m}]A_{m-1}\mor[d_{m-1}]\cdots\mor[d_{1}]A_0
\]
in a triangulated category $\cat{T}$ is an object $A$ together with a morphism
$d_0\colon A_0\to A$ such that there exists a diagram in $\cat{T}$
\[\xymatrix{A_m \ar[rr]^{d_m} \ar[dr]_{\id} \ar@{}[drr]|{\circlearrowright}
& & A_{m-1} \ar[rr]^{d_{m-1}} \ar[dr] \ar@{}[drr]|{\circlearrowright} & &
\cdots \ar[rr]^{d_2} & & A_1 \ar[rr]^{d_1} \ar[dr]
\ar@{}[drr]|{\circlearrowright} & & A_0 \ar[dr]_{d_0} \\
 & A_m \ar[ru] & & C_{m-1} \ar[ll]^{[1]} \ar[ru] & & \cdots
\ar[ll]^{[1]} & & C_1 \ar[ll]^{[1]} \ar[ru] & & A, \ar[ll]^{[1]}}\]
where the triangles with a $\circlearrowright$ are
commutative and the others are distinguished. We  point the reader in the direction of
\cite{CS,Ka,Or1} for general facts about this somewhat technical device.

\medskip

Consider the complex
\begin{equation}\label{eqn:convimpo}
R_m^\bullet:=\{R(\ko_{Y'}(-N_m))^{\oplus n_m}\boxtimes (L^{-r_m})^{\oplus s_m}\mor\cdots\mor R(\ko_{Y'}(-N_0))^{\oplus n_0}\boxtimes(L^{-r_0})^{\oplus s_0}\}.
\end{equation}
of objects in $\Db(Y \times F(Y'))$.

\begin{lem}\label{lem:exconv}
The complex $R_m^\bullet$ admits a unique (up to isomorphism) split right convolution $\kg_m=\ke_m\oplus\kf_m$ such that, for some $M<m$,  $\kh^i(\ke_m)=0$ unless $i\in[-M,0]$
and $\kh^i(\kf_m)=0$ unless $i\in[-m-M,-m]$.
\end{lem}

\begin{proof}
Due to \cite[Lemmas 2.1, 2.4]{Ka}, $R_m^\bullet$ has a unique right convolution if
\[
\Hom(R(\ko_{Y'}(-N_a))^{\oplus n_a}\boxtimes(L^{-r_a})^{\oplus s_a}, R(\ko_{Y'}(-N_b))^{\oplus n_b}\boxtimes(L^{-r_b})^{\oplus s_b}[p])=0
\]
for $a>b$ and $p<0$.

To show this, it is enough to prove, using K\"{u}nneth decomposition and the definition of the functor $R$, that, for $N,P\geq3$ and $p<0$,
$$\Hom(T_N, T_P[p])=0,$$ where $T_i:= (\epsilon'\circ\rho')(\ko_{Y'}(-i))$, $i=N,P$. Here we denote by $\epsilon':\cat{T}_{Y'}\to\Db(Y')$ the fully faithful adjoint of $\rho'$.

To this end, we first write $T_i$ as an extension of an object in $\langle\ko_{Y'}\rangle$ and an object in $\langle\cat{T}_{Y'},\ko_{Y'}\rangle$. More precisely, using the definition of semi-orthogonal decomposition, this follows performing  the following two mutations. First consider the left mutation of $\ko_{Y'}(-i)$ with respect to $\ko_{Y'}(1)$
\begin{equation}\label{eqn:tr}
\ko_{Y'}(1)^{\oplus s_i}[-3]\to\ko_{Y'}(-i)\to\widetilde{T}_i,
\end{equation}
where $s_i:=\hom^3(\ko_{Y'}(1),\ko_{Y'}(-i))$ and $\widetilde{T}_i\in\langle\cat{T}_{Y'},\ko_{Y'}\rangle$.
Applying the functor $\Hom(\ko_{Y'},-)$ to \eqref{eqn:tr}, we have that $\Hom^p(\ko_{Y'},\widetilde{T}_i)=0$, unless $p=2,3$.
Hence, the mutation with respect to $\ko_{Y'}$ gives another triangle
\begin{equation*}
\ko_{Y'}^{\oplus t^3_i}[-3]\oplus\ko_{Y'}^{\oplus t^2_i}[-2]\to\widetilde{T}_i\to T_i,
\end{equation*}
where $t^j_i:=\hom^j(\ko_{Y'},\widetilde{T}_i)$.

Assume that there is a non-zero map $\phi:T_N\to T_P[-k]$, for $k>0$.
Consider the following diagram
\begin{equation}\label{eqn:e1}
\xymatrix{
\widetilde{T}_N\ar[r]^{\alpha}&T_N\ar[d]^{\phi}\ar[r]^{\hspace{-1.5cm}\beta}&\ko_{Y'}^{\oplus t^3_N}[-2]\oplus\ko_{Y'}^{\oplus t^2_N}[-1]\\
\widetilde{T}_P[-k]\ar[r]&T_P[-k]\ar[r]^{\hspace{-2cm}\gamma}&\ko_{Y'}^{\oplus t^3_P}[-2-k]\oplus\ko_{Y'}^{\oplus t^2_P}[-1-k].
}
\end{equation}
By \eqref{eqn:tr}, one can check that $\Hom(\widetilde{T}_N,\ko_{Y'}^{\oplus t^3_P}[-2-k]\oplus\ko_{Y'}^{\oplus t^2_P}[-1-k])=0$ and so $\gamma\circ\phi\circ\alpha=0$.
Hence, we can lift $\phi$ to a morphism $\tilde{\phi}:\widetilde{T}_N\to\widetilde{T}_P[-k]$.

Consider the diagram
\[
\xymatrix{
\ko_{Y'}(-N)\ar[r]&\widetilde{T}_N\ar[d]^{\widetilde{\phi}}\ar[r]&\ko_{Y'}(1)^{\oplus s_N}[-2]\\
\ko_{Y'}(-P)[-k]\ar[r]&\widetilde{T}_P[-k]\ar[r]&\ko_{Y'}(1)^{\oplus s_P}[-2-k].
}
\]
On one hand we have
\[
\begin{split}
\Hom(\ko_{Y'}(-N),\ko_{Y'}(-P)[-k])=\Hom(\ko_{Y'}(1)^{\oplus s_N}[-2],\ko_{Y'}(1)^{\oplus s_P}[-2-k])\\=\Hom(\ko_{Y'}(-N),\ko_{Y'}(1)^{\oplus s_P}[-2-k])=0.
\end{split}
\]
On the other hand, $\Hom(\ko_{Y'}(1)^{\oplus s_N}[-2],\widetilde{T}_P[-k])=0$ by orthogonality. This gives $\widetilde{\phi}=0$.

Back to diagram \eqref{eqn:e1}, this fact and the observation that, by orthogonality, $$\Hom(\ko_{Y'}^{\oplus t^3_N}[-2]\oplus\ko_{Y'}^{\oplus t^2_N}[-1],T_P[-k])=0$$ yield $\phi=0$.

The splitting of the convolution follows from a standard argument (see, e.g., \cite[Sect.\ 4.2]{CS}).
\end{proof}

For later use, set $\tilde\ki:=\ke_m\in\Db(Y\times F(Y'))$.

\subsection{From a bijection to an isomorphism}\label{subsect:isomor}

To prove Theorem \ref{conj:Kuzne3folds}, we only need to show that the bijection induced in Proposition \ref{prop:idealsheavesandequiv} is a morphism of algebraic varieties.
Let $s$ be a closed point in $F(Y')$ and denote by $i_s:Y\times \{s\}\hookrightarrow Y\times F(Y')$ and $i'_s:Y'\times \{s\}\hookrightarrow Y'\times F(Y')$ the natural inclusions. Let $\ki'\in\coh(Y'\times F(Y'))$ be the universal ideal sheaf of lines on $Y'$.

\begin{lem}\label{lem:exconv2}
	For any closed point $s\in F(Y')$, we have $i_s^*(\tilde\ki)\cong R((i'_s)^*\ki')$.
\end{lem}

\begin{proof}
Applying the functor $i_s^*$ to the complex $R_m^\bullet$ in \eqref{eqn:convimpo}, we get the complex
\begin{equation}\label{eqn:diag2}
i^*_s(R_m^\bullet):=\{R(\ko_{Y'}(-N_m))^{\oplus n_m}\otimes\CC^{\oplus s_m}\mor\cdots\mor R(\ko_{Y'}(-N_0))^{\oplus n_0}\otimes\CC^{\oplus s_0}\}.
\end{equation}
of objects in $\Db(Y)$. It is easy to see that the objects  $i_s^*\tilde\ki\oplus i_s^*\kf_m$ (see Lemma \ref{lem:exconv}) and $R((i'_s)^*\ki')\oplus R((i'_s)^*K_m)[m]$ are both right convolutions of $i^*_s(R_m^\bullet)$.  On the other hand, the same argument as in the proof of Lemma \ref{lem:exconv}, shows that $i^*_s(R_m^\bullet)$ has a unique (up to isomorphism) right convolution and hence, by the choice of $m\gg 0$, $i_s^*(\tilde\ki)\cong R((i'_s)^*\ki')$, for any closed point $s\in F(Y')$.
\end{proof}

While $(i'_s)^*\ki'$ is the ideal sheaf $I_{l'}$ of the line $l'\subseteq Y'$ parametrized by the point $s$, the object $i_s^*(\tilde\ki)$ is nothing but the ideal sheaf $R(I_{l'})=U(I_{l'})$ of a line $l\subseteq Y$. This construction yields a morphism $F(Y')\to F(Y)$ which, being induced by the functor $R$, is actually a bijection. Hence $F(Y')\cong F(Y)$ and the proof of Theorem \ref{conj:Kuzne3folds} is complete. Notice that here we used that $F(Y)$ is smooth as recalled at the beginning of Section \ref{subsect:convuniv}.

\section{A higher dimensional example: cubic fourfolds containing a plane}\label{sec:example4folds}

The geometric interest of the non-trivial part in the semi-orthogonal decomposition of the derived category of cubic hypersurfaces can be made apparent in dimension $4$. In this section we treat the case of (generic) cubic fourfolds containing a plane proving, by completely different means, some generalization of Theorem \ref{conj:Kuzne3folds}.

To be precise, let $Y\subseteq\PP^5$ be a smooth cubic fourfold containing a
plane $P$. The projection from $P$ yields a rational map
$\pi:Y\dasharrow\PP^2$ and the blow-up of $P$ gives a quadric
fibration $\pi':\kq\to\PP^2$ whose fibres are singular along a
sextic $C\subseteq\PP^2$.
The double cover $S$ of $\PP^2$ ramified
along such a curve is a K3 surface. The cubic fourfolds containing a plane that we will study according to \cite{Kuz3} are those satisfying the following additional hypothesis:
\begin{itemize}
\item[($\ast$)] The sextic $C$ is smooth.
\end{itemize}

As observed in \cite[Rmk.\ 2.2]{MS}, by \cite[Prop.\ 1.2]{Bea}, a
smooth cubic fourfold $Y$ satisfies condition $(\ast)$ if and only if
the fibers of $\pi'$ have at most one singular point.
Notice also that the double cover $S$ is smooth as well.

One of the key geometric properties of these varieties is:

\begin{prop}\label{prop:Voisin1} {\bf (\cite{V}, Sect.\ 1, Proposition 4).}
The cubic fourfold $Y$ is determined by the sextic $C$ and an odd theta-characteristic $\theta$, i.e., a line bundle $\theta\in\Pic(C)$ such that $\theta^{\otimes 2}\cong\omega_C$ and $h^0(C,\theta)$ is odd.
\end{prop}

By \cite[Thm.\ 4.3]{Kuz3}, there exists a semi-orthogonal decomposition
\begin{eqnarray}
\Db(Y)=\langle\cat{T}_Y,\ko_Y,\ko_Y(1),\ko_Y(2)\rangle,
\end{eqnarray}
and an equivalence $\cat{T}_Y\cong\Db(S,\alpha)$, where $\alpha\in\Br(S)$ is an element in the Brauer group of the K3 surface $S$. The geometric meaning of $\alpha$ is the following. The quadric fibration $\kq$ gives a $\PP^1$-fibration $D$
over $S$ parametrizing lines contained in the fibres of $\pi'$.
The fibration $D$ is a Brauer--Severi variety and is hence
determined by the choice of an element in $\Br(S)$. Since the
fibres are projective lines, the order of $\alpha$ is $2$.

\begin{prop}\label{lem:rat}
There exist rational cubic fourfolds $Y_1$, $Y_2$ such that $\cat{T}_{Y_1}$ is not equivalent to $\cat{T}_{Y_2}$.
\end{prop}

\begin{proof}
By \cite{H2}, there exists a countable union of codimension one subvarieties in the moduli space of cubic fourfolds containing a plane consisting of rational cubic fourfolds.
Moreover, these subvarieties consist of all cubic fourfolds $Y$ containing a plane such that in the N{\'e}ron-Severi group $\NS_2(Y):=H^4(Y,\ZZ)\cap H^{2,2}(Y)$ there is a class $T$ with the property that the intersection $T\cdot Q$ is odd, where $Q$ is the class of a quadric in the fibre of $\pi:Y\dasharrow\PP^2$.

Take two such cubic fourfolds $Y_1$ and $Y_2$ with the additional requirement that the lattices $L_1$ and $L_2$ which are the saturations of $\langle H_1^2,Q_1,T_1\rangle$ and $\langle H_2^2,Q_2,T_2\rangle$ have different discriminant greater than $8$ (here $H_1$ and $H_2$ are the hyperplane sections of $Y_1$ and $Y_2$) and coincide with $\NS_2(Y_1)$ and $\NS_2(Y_2)$. Recall that the discriminant of $L_i$ is just the order of the finite group $L_i\dual/L_i$.

Let us show that $Y_1$ and $Y_2$ satisfy $(\ast)$, namely that the
singular locus of the fibres of $\pi':\kq_i\to\PP^2$ is at most one
point.
Suppose, without loss of generality, that $\kq_1$ contains a fiber $Q$ which is union of two (distinct) planes $P_1$ and $P_2$. An easy computation shows that
\[
P_i\cdot P=\frac{1}{2}(Q\cdot P)=-1\qquad\mbox{and}\qquad P_1\cdot P_2=\frac{1}{2}(Q^2-P_1^2-P_2^2)=-1,
\]
for $i\in\{1,2\}$. In particular, $P$, $P_1$ and $P_2$ are distinct classes in $\NS_2(Y_1)$ and the sublattice $N$ of $\NS_2(Y_1)$ which is the saturation of the lattice $\langle H^2, P, P_1,P_2\rangle$ has rank bigger than $3$, contradicting the choice of $Y_1$. The case where a fibre degenerates to a double plane is similar and left to the reader.

Since $Y_i$ is rational, by \cite[Prop.\ 4.7]{Kuz3}, $\cat{T}_{Y_i}$ is equivalent to $\Db(S_i)$ ($i=1,2$), and hence we need to prove that $\Db(S_1)\not\cong\Db(S_2)$. A result of Orlov (\cite{Or1}) shows that this happens if and only if the transcendental lattices $T(S_1):=\Pic(S_1)^\perp$ and $T(S_2):=\Pic(S_2)^\perp$ are not Hodge isometric. By the results in \cite[Sect.\ 1]{V}, the lattice $T(S_i)$ has the same discriminant as $L_i$.  Since these discriminants are different, $T(S_1)\not\cong T(S_2)$.
\end{proof}

Recall that a cubic fourfold $Y$ containing a plane $P$ is \emph{generic} if the
group of codimension-$2$ algebraic classes $\NS_2(Y)\subseteq
H^4(Y,\ZZ)$ is generated by the class of $P$ and by $H^2$, where
$H\in H^2(Y,\ZZ)$ is the class of a hyperplane section of $Y$. By the calculations in the proof of Proposition \ref{lem:rat}, these fourfolds satisfy condition $(\ast)$.
Kuznetsov's Conjecture \ref{conj:Kuzne4folds} predicts that a generic cubic fourfold $Y$ containing a plane $P$ is
not rational since $\Db(S,\alpha)$ is not equivalent to the
derived category of any un-twisted K3 surface (see \cite[Prop.\
4.8]{Kuz3}).

The following result gives an analogue for cubic fourfolds containing a plane of the Torelli theorem for cubic threefolds proved in this paper.

\begin{prop}\label{prop:conjgen}
Given a cubic fourfold $Y$ containing a plane $P$ and satisfying $(\ast)$, there exist only finitely many isomorphism classes of cubic fourfolds $Y_1=Y,Y_2,\ldots,Y_n$ containing a plane and with the property $(\ast)$ such that $\cat{T}_Y\cong\cat{T}_{Y_j}$, with $j\in\{1,\ldots,n\}$.
Moreover, if $Y$ is generic, then $n=1$.
\end{prop}

\begin{proof}
Let $Y'$ be a cubic fourfold such that $\cat{T}_{Y'}\cong\cat{T}_{Y}$ and containing a plane $P'$ giving an equivalence $\cat{T}_{Y'}\cong\Db(S',\alpha')$ and hence an equivalence
$\Phi:\Db(S,\alpha)\isomor\Db(S',\alpha')$. By \cite[Thm.\
0.4]{HS1}, $\Phi$ induces a Hodge isometry
$\Phi^*:T(S,\alpha)\isomor T(S',\alpha')$, where $T(S,\alpha)$ and $T(S',\alpha')$ are the generalized
transcendental lattices. To be precise, these lattices depend on the choice of $B$-field lifts $B$ and $B'$ of the Brauer classes $\alpha$ and $\alpha'$.

By \cite[Sect.\ 1]{V} (see, in particular, \cite[Sect.\ 1, Prop.\ 3
]{V}), the weight-$2$ Hodge structure on $T(S,\alpha)$ determines the Hodge structure on $H^4(Y,\ZZ)$, since $T(S,\alpha)(-1)$ is realized as a primitive sublattice of the orthogonal complement $L:=\langle H^2,P\rangle^\perp$ in $H^4(Y,\ZZ)$. Again $H$ and $P$ are respectively the hyperplane section of $Y$ and the plane contained in $Y$. (Recall that, given a lattice $L$ with quadratic form $b_L$, the lattice $L(-1)$ coincides with $L$ as a group but its quadratic form $b_{L(-1)}$ is such that $b_{L(-1)}=-b_L$.)

This means that, by the Torelli theorem for cubic fourfolds (see \cite{V}), we are reduced to proving that there are only finitely many primitive sublattices $T$ of the sublattice $L:=\langle H^2,P\rangle^\perp$ in $H^4(Y,\ZZ)$ with any isometry $\varphi:T(S,\alpha)(-1)\isomor T$ which does not extend to an isometry $\overline{\varphi}:L\isomor L$, fixing the class $H^2$. But this is a standard result which can be found, for example, in the proof of \cite[Cor.\ 4.6]{HS1}.

For the second part of the statement, assume there exists an equivalence
$\Phi:\Db(S_1,\alpha_1)\isomor\Db(S_2,\alpha_2)$ inducing, as before, a Hodge isometry
$\Phi^*:T(S_1,\alpha_1)\isomor T(S_2,\alpha_2)$. Take a B-field lift $B_i$ for $\alpha_i$. (See \cite{HS1} for more details.)

We want to show that there is a Hodge isometry $f:T(S_1)\to
T(S_2)$ making the following diagram commutative:
\begin{eqnarray}\label{eq:dia}
    \xymatrix{
    0\ar[r]&T(S_1,\alpha_1)\ar[d]_{\Phi^*}\ar[r]^{\;\;\;i_1}&T(S_1)\ar@{.>}[d]_{f}\ar[r]^{\wedge B_1}&\ZZ/2\ZZ\ar[r]\ar@{=}[d]&0\\
    0\ar[r]&T(S_2,\alpha_2)\ar[r]^{\;\;\;i_2}&T(S_2)\ar[r]^{\wedge B_2}&\ZZ/2\ZZ\ar[r]&0.
    }
\end{eqnarray}
First of all, observe that, up to considering the composition
$i_1':=i_2\circ\Phi^*:T(S_1,\alpha_1)\hookrightarrow T(S_2)$, we can
assume $T(S_2,\alpha_2)=T(S_1,\alpha_1)$. Thus, let $\tau\in
T(S_1,\alpha_1)\otimes\QQ$ be such that $i_1(\tau)$ generates $T(S_1)$
modulo $T(S_1,\alpha_1)$. Obviously, $\tau':=2\tau\in i_1(T(S_1,\alpha_1))$.
Define $f$ by
$$f(i_1(\tau)):=\frac{1}{2}(i_1'(i_1^{-1}(2\tau)))\qquad\mbox{and}\qquad f(t):=i_2(\Phi^*(i_1^{-1}(t))),$$ for any $t\in i_1(T(S_1,\alpha_1))$.

The morphism $f$ is obviously an isometry, since the $\QQ$-linear
extension of $i_1'\circ i_1^{-1}$ is. For the same reason, $f$
preserves the weight-$2$ Hodge structure on $T(S_1)$ and $T(S_2)$.

Since $S_1$ and $S_2$ are generic K3 surfaces with $\Pic(S_i)\cong\ZZ$
generated by an element with self-intersection $2$, the isometry
$f$ extends to a Hodge isometry $$f':H^2(S_1,\ZZ)\isomor
H^2(S_2,\ZZ).$$ Up to composing $f'$ with $-\id$ and changing
$B_2$ with $-B_2$, there exists an isomorphism $\varphi:S_1\isomor
S_2$ such that $f'=\varphi_*$. (Notice that changing $B_2$ with
$-B_2$ is no problem since $\exp(B_2)=\exp(-B_2)=\alpha_2$ as
$\alpha_2$ has order $2$.) In other words, by \eqref{eq:dia},
$\varphi^*(\alpha_2)=\alpha_1$.

For $i\in\{1,2\}$, define, as before, $L_i:=\langle
H_i^2,P_i\rangle^{\perp_{H^4(Y_i,\ZZ)}}$, where $H_i$ and $P_i$
are, respectively, the classes of an hyperplane section of $Y_i$
and of the plane contained in $Y_i$. Again, by the discussion in
\cite[Sect.\ 1]{V} (see, in particular, \cite[Sect.\ 1, Prop.\ 3]{V}), there exists a short exact sequence
\begin{eqnarray}\label{eq:ext}
    \xymatrix{
    0\ar[r]&L_i(-1)\ar[r]&T(S_i)\ar[r]^{\wedge B_i}&\ZZ/2\ZZ\ar[r]&0.
    }
\end{eqnarray}

In \cite[Sect.\ 2]{V}, it is shown that, given $Y_i$ and $S_i$,
there exists a natural isomorphism between the affine space (over
$\ZZ/2\ZZ)$ of the theta-characteristics on the sextic $C_i$ along
which the double cover $S_i$ of $\PP^2$ ramifies and the extension
classes as in \eqref{eq:ext}. In particular, the isomorphism
$\varphi:S_1\isomor S_2$ leads to isomorphic sextics and
theta-characteristics. Applying Proposition \ref{prop:Voisin1}, we
get the desired isomorphism $Y_1\cong Y_2$.
\end{proof}

\begin{remark}
For non-generic cubic fourfolds containing a plane one cannot expect that the derived category $\Db(S,\alpha)$ determines the fourfold $Y$ up to isomorphism. Indeed, using the properties of the moduli space of cubic fourfolds in \cite{H1}, it is possible to construct examples of fourfolds $Y_1$ and $Y_2$ with a Hodge isometry $T(S_1,\alpha_1)\cong T(S_2,\alpha_2)$ but such that $L_1:=\langle H_1^2,P_1\rangle^\perp$ and $L_2:=\langle H_2^2,P_2\rangle^\perp$ are not isometric.
\end{remark}

\medskip

{\small\noindent{\bf Acknowledgements.}
It is a great pleasure for the authors to thank Alexander Kuznetsov for stating the problem, giving them ideas how to solve it, kindly answering their questions, and proposing several simplifications in the original arguments. They are also grateful to Daniel Huybrehcts for his comments on a first draft of the paper.
E.~M.~ wants to thank Arend Bayer for numerous discussions, comments, and explanations while M.~B.~ is grateful to Arnaud Beauville for pointing out this problem. P.~S.~ was partially supported by the Istituto Nazionale di Alta Matematica while M.~B.~ was supported by the SFB/TR 45.
E.~M.~ is partially supported by the NSF grants DMS-1001482 and DMS-1160466, the Hausdorff Center for Mathematics, Bonn, and SFB/TR 45.
The authors thank the Department of Mathematics of the University of Utah for the hospitality during the preparation of this paper.}



\begin{thebibliography}{99}


\bibitem{Bea} A.\ Beauville, \emph{Vari\'et\'es de Prym et jacobiennes interm\'ediaires}, Ann.\ Sci.\ \'Ec. Norm.\ Sup.\ {\bf 10} (1977), 309--391.

\bibitem{Bondal} A.I.\ Bondal, \emph{Representations of associative algebras and coherent sheaves}, Math. USSR-Izv. {\bf 34} (1990), 23--42.

\bibitem{BO2} A.\ Bondal, D.\ Orlov, \emph{Semiorthogonal decomposition for algebraic varieties}, arXiv:alg-geom/9506012.

\bibitem{Br} T.\ Bridgeland, \emph{Stability conditions on triangulated categories}, Ann.\ Math.\ {\bf 166} (2007), 317--346.

\bibitem{Br1} T.\ Bridgeland, \emph{Stability conditions on K3 surfaces}, Duke Math.\ J.\ {\bf 141} (2008), 241--291.

\bibitem{CP} J. Collins, A. Polishchuk, \emph{Gluing stability conditions}, Adv.\ Theor.\ Math.\ Phys.\ {\bf 14}, No. 2  (2010), 563--607.

\bibitem{CS} A.\ Canonaco, P.\ Stellari, \emph{Twisted Fourier--Mukai functors}, Adv.\ Math.\ {\bf 212} (2007), 484--503.

\bibitem{CG} C.H.\ Clemens, P.\ Griffiths, \emph{The intermediate Jacobian of the cubic threefold}, Ann.\ Math.\ {\bf 95} (1972), 281--356.

\bibitem{HRS} D.\ Happel, I.\ Reiten, S.O.\ Smal\o, \emph{Tilting in abelian categories and quasitilted algebras}, Mem. Amer. Math. Soc. {\bf 120} (1996).

\bibitem{H1}  B.\ Hassett, {\it Some rational cubic fourfolds}, J. Alg. Geom. {\bf 8} (1999), 103--114.

\bibitem{H2}  B.\ Hassett, {\it Special cubic fourfolds}, Compositio Math. {\bf 120} (2000), 1--23.

\bibitem{Huy} D.\ Huybrechts, \emph{Fourier--Mukai transforms in algebraic geometry}, Oxford Mathematical Monographs (2006).

\bibitem{HL} D.\ Huybrechts, M.\ Lehn, \emph{The geometry of moduli spaces of sheaves}, Aspects of Math.\ E31, Vieweg, Braunschweig, (1997).

\bibitem{HS1} D.\ Huybrechts, P.\ Stellari, \emph{Equivalences of twisted K3 surfaces}, Math. Ann. {\bf 332} (2005), 901--936.

\bibitem{Ka} Y. Kawamata, \emph{Equivalences of derived categories of sheaves on smooth stacks}, Am. J. Math. {\bf 126} (2004), 1057--1083.

\bibitem{Kuz1} A.\ Kuznetsov, \emph{Derived category of a cubic threefold and the variety $V\sb {14}$}, Proc.\ Steklov Inst.\ Math.\ {\bf 246} (2004), 171--194.

\bibitem{KuzHPD} A.\ Kuznetsov, \emph{Homological projective duality}, Publ.\ Math.\ Inst.\ Hautes \'{E}tudes Sci.\ {\bf 105} (2007), 157--220.

\bibitem{Kuz2} A.\ Kuznetsov, \emph{Derived categories of quadric fibrations and intersections of quadrics}, Adv. Math. {\bf 218} (2008), 1340--1369.


\bibitem{Kuz3} A.\ Kuznetsov, \emph{Derived categories of cubic fourfolds}, in \emph{Cohomological and geometric approaches to rationality problems}, Bogomolov, Fedor (ed.) et al.,
New Perspectives, Boston, MA: Birkh\"auser. Progress in Mathematics {\bf 282} (2010), 219--243.

\bibitem{Kuz5} A.\ Kuznetsov, \emph{Derived categories of Fano threefolds}, arXiv:0809.0225.


\bibitem{Inaba} M.\ Inaba, \emph{Toward a definition of moduli of complexes of coherent sheaves on a projective scheme}, J.\ Math.\ Kyoto Univ.\ {\bf 42} (2002), 317--329.

\bibitem{Lieb} M.\ Lieblich, \emph{Moduli of complexes on a proper
    morphism}, J.\ Alg.\ Geom.\ {\bf 15} (2006), 175--206.

\bibitem{MS} E.\ Macr\`i, P.\ Stellari, \emph{Fano varieties of cubic
    fourfolds containing a plane}, arXiv:0909.2725.

\bibitem{Okada} S. Okada, \emph{On stability manifolds of Calabi--Yau surfaces}, Int.\ Math.\ Res.\ Not.\ (2006), Art.\ ID 58743, 16 pp.

\bibitem{Orlov} D.\ Orlov, \emph{Projective bundles, monoidal transformations, and derived categories of coherent sheaves}, Russian Acad.\ Sci.\ Izv.\ Math.\ {\bf 41} (1993), 133--141.

\bibitem{Or1} D.\ Orlov, \emph{Equivalences of derived categories and K3 surfaces}, J. Math. Sci. {\bf 84} (1997), 1361--1381.

\bibitem{Simpson} C.T.\ Simpson, \emph{Moduli of representations of the fundamental group of a smooth projective variety I}, Publ.\ Math.\ IHES {\bf 79} (1994), 47--129.

\bibitem{V} C.\ Voisin, \emph{Th\'{e}or\`{e}me de Torelli pour les cubiques de ${\mathbb{P}}^5$}, Invent. Math. {\bf 86} (1986), 577--601.

\end{thebibliography}
\end{document}